\documentclass{article}
	\usepackage{amsthm,amsmath,stmaryrd,bbm,amssymb, mathrsfs, amsbsy, dsfont, mathabx}
	\usepackage[usenames,dvipsnames]{xcolor}
	\usepackage[english]{babel}
	\usepackage[utf8]{inputenc}
	\usepackage[scale = .8]{geometry}
	
	\usepackage{graphicx}
	\usepackage{verbatim}
	\usepackage{enumitem}
	\usepackage{etoolbox}
	\usepackage[colorlinks=true,linkcolor=MidnightBlue, citecolor=MidnightBlue]{hyperref}
	\usepackage{todonotes}
	
	\newcommand{\R}{\mathbb R}

	\newcommand{\N}{\mathbb N}

	\newcommand{\CC}{\mathcal C}
	
	\newcommand{\EE}{\mathcal E}
	\newcommand{\FF}{\mathcal F}
	\newcommand{\GG}{\mathcal G}
	\newcommand{\HH}{\mathcal H}
	
	\newcommand{\KK}{\mathcal K}

	\renewcommand{\SS}{\mathcal S}
	\newcommand{\TT}{\mathcal T}

	\newcommand{\bHH}{ \boldsymbol{\mathcal{H}} }
	\newcommand{\bXX}{ \boldsymbol{\mathcal{X}} }

	\theoremstyle{plain}
	\newtheorem{theo}{Theorem}
	\newtheorem{prop}[theo]{Proposition}
	\newtheorem{lem}[theo]{Lemma}
	\newtheorem{cor}[theo]{Corollary}
	\theoremstyle{remark}
	\newtheorem{rem}[theo]{Remark}
	\newtheorem{exam}[theo]{Example}
	\theoremstyle{definition}
	
	\newtheorem{assu}[theo]{Assumption}
	
	\numberwithin{equation}{section}
	\numberwithin{theo}{section}
	
	\def\le{\leqslant}
	\def\ge{\geqslant}

	\DeclareMathOperator{\id}{Id}
	
	\def\eps{\varepsilon}
	\renewcommand\d{\textnormal{d}}
	\def\ini{\textnormal{in}}
	\newcommand{\di}{\mathrm{i}}
	
	\def\la{\langle}
	\def\ra{\rangle}
	\makeatletter
	\newsavebox{\@brx}
	\newcommand{\lla}[1][]{\savebox{\@brx}{\(\m@th{#1\langle}\)}%
	\mathopen{\copy\@brx\kern-0.5\wd\@brx\usebox{\@brx}}}
	\newcommand{\rra}[1][]{\savebox{\@brx}{\(\m@th{#1\rangle}\)}%
	\mathclose{\copy\@brx\kern-0.5\wd\@brx\usebox{\@brx}}}
	\makeatother
	\newcommand{\Nt}[1]{{\left\vert\kern-0.25ex\left\vert\kern-0.25ex\left\vert #1 
	\right\vert\kern-0.25ex\right\vert\kern-0.25ex\right\vert}}
	
	\newcommand{\step}[2]{\medskip\noindent\textit{Step #1: #2.}}

	\makeatletter
	\newcounter{author}
	\renewcommand*\author[1]{%
	\stepcounter{author}%
	\ifnum\c@author=1
	\gdef\@author{#1}%
	\else
	\xdef\@author{\unexpanded\expandafter{\@author\and#1}}%
	\fi
	\csgdef{author@\the\c@author}{#1}}
	\newcommand*\email[1]{%
	\csgdef{email@\the\c@author}{#1}}
	\newcommand*\address[1]{%
	\csgdef{address@\the\c@author}{#1}}
	\AtEndDocument{%
	\xdef\author@count{\the\c@author}%
	\c@author=1
	\print@authors}
	\newcommand*\print@authors{%
	\ifnum\c@author>\author@count
	\else
	\print@author{\the\c@author}%
	\advance\c@author by 1
	\expandafter\print@authors
	\fi}
	\newcommand*\print@author[1]{%
	\par\medskip
	
	\noindent\begin{tabular}{@{}l@{}}%
	\textsc{\csuse{author@#1}}\\[.25em]
	\begin{minipage}{\textwidth}\csuse{address@#1}
	\end{minipage}\\[.75em]
	\textit{E-mail}:
	\href{mailto:\csuse{email@#1}}{\csuse{email@#1}}
	\end{tabular}
	}
	\makeatother

	\title{Well-posedness and long-time behavior for self-consistent Vlasov-Fokker-Planck equations with general potentials}

	\author{Pierre Gervais}
	\address{Université de Toulouse, Institut de Mathématiques de Toulouse, \\
	F-31062 Toulouse, France} 
	\email{pierre.gervais@math.univ-toulouse.fr}
	
	\author{Maxime Herda}
	\address{Univ. Lille, CNRS, Inria, UMR 8524 - Laboratoire Paul Painlevé,\\ F-59000 Lille, France} 
	\email{maxime.herda@inria.fr}
	
\begin{document}
	
	\maketitle
	
	\begin{abstract}
	
	We study the well-posedness, steady states and long time behavior of solutions to  Vlasov-Fokker-Planck equation with external confinement potential and nonlinear self-consistent interactions. Our analysis introduces newly characterized conditions on the interaction kernel that ensure the local asymptotic stability of the unique steady state.  Compared to previous works on this topic, our results allow for large, singular and non-symmetric interactions. As a corollary of our main results, we show exponential decay of solutions to the Vlasov-Poisson-Fokker-Planck equation in dimension $3$, for low regularity initial data. In the repulsive case, the result holds in strongly nonlinear regimes (\emph{i.e.} for arbitrarily small Debye length). Our techniques rely on the design of new Lyapunov functionals based on hypocoercivity and hypoellipticity theories. We use norms which include part of the interaction kernel, and carefully mix ‘‘macroscopic quantities based''--hypocoercivity with ‘‘commutators based''--hypocoercivity.

	\medskip
	
	
	\medskip
	
	\end{abstract}
	
	\tableofcontents
	
	\section{Introduction}
	We are interested in the study of Vlasov-Fokker-Planck (VFP) equation in the presence of an external confining potential and self-consistent interactions. This equation describes the time evolution of the  distribution function $F$ in the phase space $\R^d\times\R^d$ with $d\geq1$, through the kinetic equation
	\begin{equation}\label{eq:SCVFP}
	\begin{cases}
	\displaystyle \partial_t F + v \cdot \nabla_x F - \nabla_x \left( \Psi_F + V \right) \cdot \nabla_v F = \nu \nabla_v \cdot \left( v F + \nabla_v F \right),
	\\
	\displaystyle \Psi_F (t, x) = \int_{\R^{2d} } k (x - y) F(t, y, v)\, \d y\, \d v, \\
	\displaystyle F|_{t=0} = F_\ini.
	\end{cases}
	\end{equation}
	In this equation the Vlasov transport operator $v \cdot \nabla_x -\nabla_x \left( \Psi_F + V \right) \cdot \nabla_v$ accounts for the movement of particles in the phase space under the action of two forces. The first force is driven by an external potential $V$ with confining properties. The second force is self-consistent and derives from the potential $\Psi_F$ which is given by the convolution between a long-range interaction kernel $k$ and the macroscopic density. This makes the model \eqref{eq:SCVFP} nonlinear and nonlocal. The Fokker-Planck operator $\nu\nabla_v \cdot ( v\ \cdot\ ) + \nu\Delta_v$, with $\nu>0$, accounts for short range interactions with,  typically, a fixed background of particles at constant temperature. In appropriate contexts, it can also be viewed as a toy model for nonlinear collisional operators such as the Dougherty-Fokker-Planck \cite{dougherty1964model} or Landau operator \cite{villani2002review}. The Vlasov-Fokker-Planck equation can also be interpreted as the Kolmogorov-Fokker-Planck equation for the law of a stochastic process following the underdamped Langevin equations, which describes the trajectories of individual particles.

	The self-consistent Vlasov-Fokker-Planck equation appears in a variety of physical contexts ranging from astrophysics \cite{chandrasekhar_1943_stochastic},  plasmas \cite{bouchut_1995_long, herda_2018_large}, relativistic beams \cite{cesbron2023vlasov} or chemical solutions \cite{wu2015diffusion}. It even has applications in machine learning \cite{cheng2018underdamped}, through its relation with the underdamped Langevin process. Depending on the physics at hand the interaction potential may take various forms. A widely studied model is the Vlasov-Poisson-Fokker-Planck (VPFP) equation for which the kernel $k$ is given by the Coulomb kernel 
	\begin{equation}\label{eq:kCoulomb}
	k_C(x) = \frac{I}{|x|^{d-2}},
	\end{equation} 
	with $d=3$. The potential $k_C$ describes repulsive electrostatic interactions with  $I>0$ scaling as the squared inverse Debye length. 
	In the case of attractive gravitational interactions, the Newton potential 
	\begin{equation}\label{eq:kNewton}
	k_N(x) = -\frac{I}{|x|^{d-2}},
	\end{equation} 
	with $d=3$, is used. The Coulomb and Newton interaction kernels are radially symmetric and in particular even functions of their arguments, which is natural for particle systems arising from classical mechanics. However, for other applications, asymmetric interaction kernels can arise.

	In particle accelerator physics, the  VFP equation \eqref{eq:SCVFP} is used to describe the longitudinal dynamics of relativistic beams of charged particles. In this context, there are several  models for describing self-consistent interactions, taking into account synchrotron radiation effects. In the simplest case of a relativistic particle in free space on a circular orbit the interaction potential (related to the wakefield in this context) is given by 
	\begin{equation}\label{eq:kSynchrotron}
	k_S(x) = 2\frac{\cosh\left(\frac53\sinh^{-1}(x)\right)-\cosh\left(\sinh^{-1}(x)\right)}{\sinh\left(2\sinh^{-1}(x)\right)}\mathds{1}_{x>0}(x),
	\end{equation} 
	with $d=1$. We refer to \cite{cesbron2023vlasov} and references therein for further details on this model. Observe that unlike $k_C$ and $k_N$, the potential $k_S$ is not an even function of its argument (see \cite[Figure 2]{cesbron2023vlasov} for a graphical representation).
	
	\paragraph{Steady state} The steady states solutions of the Vlasov-Fokker-Planck equation are of the form (see \cite{dressler, dolbeault_1991_stationary, bouchut_1995_long} or Section~\ref{scn:steady_state})
	$$F_\star(x, v) = \rho_\star(x) M(v), \quad \rho_\star(x) = e^{-V_\star(x)},$$
	where the local Maxwellian distribution is given by 
	$$M(v) = \frac{e^{-\frac{|v|^2}{2}}}{(2\pi)^{\frac d2}},$$
	and the Gibbs potential $V_\star$ is such that
	$$
	\quad \nabla_x V_\star = \nabla_x \left( V + k\ast\rho_\star \right),	
	$$
	where $\ast$ denotes the convolution. Observe that $k\ast\rho_\star = \Psi_{F_\star}$. The macroscopic part of the steady state can be characterized equivalently, under the normalization condition $ \int \rho_\star = 1$, as the fixed point of some nonlinear integral mapping
	$$\rho_\star = \frac{ e^{-V - k\ast\rho_\star} }{ \int_{ \R^d } e^{-V - k\ast\rho_\star } \d x }.$$
	This steady state equation coincides with the one of the McKean-Vlasov equation
	$$\partial_t \rho = \nabla_x \cdot \left( \nabla_x \rho + \rho \nabla_x k\ast\rho + \rho \nabla_x V \right).$$
	This equation can be seen as the macroscopic version of \eqref{eq:SCVFP} and can be, at least formally, related to  \eqref{eq:SCVFP} in the diffusion limit $\nu\to\infty$ on the appropriate time scale. We refer to the recent  \cite{blaustein_2023_diffusive} and references therein for more details on diffusion limits, at least in the case of Coulomb interactions. We point out that the existence and stability of steady states for the McKean-Vlasov equation has been considered for instance in \cite{tamura, chazelle, carrillo_2020_long}, as well as the question of bifurcations and phase transitions.

	\paragraph{Free energy}
	In some cases there is a macroscopic quantity of interest (aside of the total mass) for the Vlasov-Fokker-Planck equation \eqref{eq:SCVFP}. It is usually refered to as free energy functional and writes
	\begin{equation}
	\label{eq:freeenerg}
	\EE[ F ] = \int_{ \R^{2d} } F(x,v) \log F(x,v) \, \d x \d v + \int_{ \R^{2d} } F(x, v) \left( \frac{|v|^2}{2} + V(x) + \frac{ \Psi_F(x) }{2} \right) \, \d x \d v,
	\end{equation}
	where the first part represents the entropy of the system, and the second one represents the sum of its kinetic and potential energy. One may check (see for instance \cite{cesbron2023vlasov}) that this functional satisfies the differential identity
	\begin{equation*}
	\frac{\d \EE}{\d t} + \nu \int_{ \R^{2d} } F(x, v) \left| \nabla_v \log\left( \frac{F(x,v)}{M(v)} \right) \right|^2 \d x \d v = - \int_{ \R^{2d} } \nabla_x \Psi^o_F(x) v F(x, v) \d x\d v \, ,
	\end{equation*}
	where $\Psi_F^o$ is the part of the total interaction potential induced by the odd part of $k$, namely $k^\text{o}(x) = (k(x)-k(-x))/2$ (see Section \ref{scn:interaction}).	In particular, when $k$ is even, the right-hand side of the previous identity vanishes, thus the free energy is a Lyapunov functional. In this case the steady states of VFP described previously coincide with critical points of $\EE$ under the constraint of given total mass only if $k$ is even. When $k$ yields enough coercivity and convexity to the functional $\EE$, it becomes clear that the steady state is actually unique and asymptotically stable. This is the case for instance for Coulomb interactions, for which \eqref{eq:freeenerg} makes it possible to show non-quantitative asymptotic stability of the dynamics, see \cite{bouchut_1995_long}. 	
	
	However, for asymmetric $k$, the functional $\EE$ does not yield as much information anymore. We shall see in Section~\ref{scn:steady_state} that it may still be used to investigate the uniqueness of steady states. For more details on Lyapunov functionals for \eqref{eq:SCVFP} with general asymmetric kernels we refer to \cite{cesbron2023vlasov} and \cite{monmarche2023note}.
	
	\paragraph{State of the art and contributions}
	
	In the present paper, we study the well-posedness of self-consistent VFP equations \eqref{eq:SCVFP} around steady states, which we prove to exist and be unique under appropriate conditions, and derive quantitative estimates of decay and regularization of solutions.
	
	Early works on the well-posedness of VFP include \cite{neunzert_1984_VFP, degond_1986_global, victory_1990_classical, victory_1991_existence, bouchut_1993_existence}. Concerning the long-time behavior, in linear settings ($k=0$), quantitative convergence estimates can be established thanks to hypocoercivity methods \cite{herau_2004_isotropic, villani_2009_hypocoercivity, DMS, bouin_2020_hypocoercivity}. In the nonlinear case, with smooth interaction kernels ($k\in W^{2,\infty}$), hypocoercivity methods have been complemented with techniques involving entropy / free energy functionals \cite[Section A.21]{villani_2009_hypocoercivity}, low regularity norms and probabilistic techniques \cite{bolley_2010_trend, guillin_2022_convergence, bayraktar_2024_exponential}. 
	
	The case of singular kernels and more precisely the Coulomb interactions, and correspondingly the Vlasov-Poisson-Fokker-Planck equation has been studied a lot. Notably, \cite{herau_2016_global} provided the first quantitative decay estimates for VPFP with confining potential in dimension $3$, though their results were restricted to weakly nonlinear regimes (for $k$ sufficiently small). Subsequent work \cite{herda_2018_large} continued in this direction. Recently, \cite{ADLT} demonstrated that appropriate norms could extend these results beyond weak interaction regimes for the linearized system, leaving the nonlinear case as a future research direction. Improvements were obtained recently in this direction by \cite{arnold_tosh_VFP}. Asymmetric kernels have only recently been explored in \cite{cesbron2023vlasov} and \cite{monmarche2023note}, with results limited to one dimension and requiring at least Lipschitz regularity for $k$. Let us finally mention a different regime, in which the collision frequency $\nu$ is small and the confinement is replaced by a periodic box. There, the interplay between phase mixing and collisions produces \emph{enhanced dissipation}: the non-zero spatial modes relax on the time scale $\nu^{-1/3}$, much shorter than the thermalization scale $\nu^{-1}$, and this mechanism can be combined with Landau damping to obtain nonlinear stability results. We refer to \cite{bedrossian_2017_suppression, bedrossian_cotizelati_2017} and references therein. Such phenomena are of a different nature from the confinement-driven relaxation studied here, but they suggest that the dependence of our rates on $\nu$ is presumably not optimal for small $\nu$.
	
	In light of the existing literature, we have identified several gaps that this manuscript  starts to address. First, we enhance state-of-the-art results by deriving quantitative decay estimates for VPFP in strongly nonlinear regimes. This goes beyond the restrictions of \cite{herau_2016_global, herda_2018_large} and addresses part of the open problems posed in \cite{ADLT}. We also improve the regularity requirements on the initial data, compared to the recent \cite{arnold_tosh_VFP}. Furthermore, our analysis goes beyond VPFP and includes a large class of VFP models. Indeed we place particular emphasis on considering general interaction potentials $k$, which may be non-symmetric, singular, or exhibit repulsive and/or mildly attractive behavior. In this regard we establish sufficient criteria for the uniqueness of steady states and their asymptotic stability, noting that such stability is not always guaranteed.
	
	In the following section, we precisely define our setting and present our main results.

	\section{Setting and main result}

	In this section, we present the main results of this paper. First we introduce some notation	and the general assumptions, under which our main results hold.

	\subsection{Notations}
	
	In the following, for a complex number $z$, the real and imaginary parts are denoted by $\Re z$ and $\Im z$. The notations $x_+ = \max(0,x)$ and $x_-=\max(0,-x)$ denote the positive and negative parts respectively, so that $x = x_+ - x_-$. Given a multi-index $\alpha=(\alpha_1,\dots,\alpha_d)\in\N^d$, we write $\partial^\alpha$ for the partial derivative $\partial_{x_1}^{\alpha_1}\cdots\partial_{x_d}^{\alpha_d}$ and $|\alpha| = \sum_{i}\alpha_i$. Given a function (or a distribution) $k:\R^d\to \R$, we write $\widecheck{k}:x\mapsto  k(-x)$ and we denote by $\FF(k)(\xi) = \widehat{k}(\xi) = \int_{\R^d}e^{-i x\cdot\xi}k(x)\d x$ the Fourier transform of $k$. The symbol $\ast$ denotes the convolution between a distribution and a function, namely
	$
	k\ast\rho(x) = \int_{\R^{d}} k(x-y) \rho(y)\, \d y
	$.  
	
	The notation $\la\cdot,\cdot\ra$ denotes a duality bracket or a scalar product. In the latter case, a subscript may be added to precise the Hilbert space. We write $[A,B] =AB-BA$ to denote the commutator between two operators. Given $p\in[1,\infty]$, we denote by $p'=(1-p^{-1})^{-1}$ the conjugate Lebesgue exponent. For weighted Lebesgue space, we denote the norms
	$$ \| f \|_{L^p(m)}^p = \int_{\Omega} | f (\omega) |^p m(\omega) \d \omega \, , \quad 1 \le p <\infty, \qquad \| f \|_{L^\infty(m)} = \sup_{\omega \in \R^d} | f (\omega)  m(\omega) | \, .$$
	For weighted Sobolev spaces, when $s \in \N$, we denote the squared norm
	$$\| f \|_{H^s(m)}^2 = \sum_{ | \alpha | \le s  } \| \partial^\alpha f \|_{L^2(m)}^2$$
	and when $s \in (0, \infty) \backslash \N$, we define $H^s(m)$ by interpolation (through the real or complex method unambiguously, see \cite[Theorems 3.3 and 3.5]{chandler}). 
	
	In inequalities $C$ will denote a constant which may change from one line to the next and we sometimes use $\lesssim$ instead of $\leq C$. When it is necessary, we shall be more precise and write $C(a,b,c,\dots)$ to denote a constant depending on the parameters $a,b,c,\dots$. 
	
	\subsection{Confining potential}
	Let us first describe our general assumptions on the confinement potential.	
	\begin{assu}
	\label{assu:confinement}
	The confining potential $V:\R^d\to\R$ is a $C^2(\R^d)$ function. It is assumed to satisfy that for any $\eps > 0$ there is some constant $C_\eps > 0$, such that
	\begin{equation}
	\label{eq:smoothness_V}
	\forall x \in \R^d, \quad | \nabla^2 V(x) | \le \eps | \nabla V(x) | + C_\eps.
	\end{equation}
	Moreover the following integrability and boundedness conditions hold
	\begin{equation}
	\label{eq:integrability_V}
	\left( 1 + | \nabla V |^2 \right) e^{-V} \in L^1 \cap L^\infty  \quad \text{and} \quad \int_{\R^d} e^{-V(x) } \d x = 1 \,.
	\end{equation}
	Finally, we assume that the measure $e^{-V}$ admits a Poincaré inequality
	\begin{equation}
	\label{eq:Poincare_V}\int_{\R^d} | u|^2 e^{- V}\d x - \left(\int_{\R^d} u e^{- V} \d x\right)^{2} \le C_P \int_{\R^d} | \nabla_x u |^2 e^{- V} \d x \, ,
	\end{equation}
	for any $u:\R^d\to\R$ such that the integrals are finite and some constant $C_P>0$.
	\end{assu}
	
	The Poincaré inequality \eqref{eq:Poincare_V} is equivalent to the existence of a spectral gap for the Witten Laplacian associated to $V$, namely $- \Delta_V = -\Delta + \frac{1}{4} | \nabla V|^2 - \frac{1}{2} \Delta V$. 
	A sufficient condition for $-\Delta_V$ to admit a spectral gap is given by Persson's Theorem \cite[Theorem 2.1]{persson} taking $a(x) = \frac{1}{4} | \nabla V(x) |^2 - \frac{1}{2} \Delta V(x)$, which requires that
	$$\inf_{|x| \ge R} a(x) > 0 \quad \text{for $R$ large enough} \, .$$
	By virtue of \eqref{eq:smoothness_V}, the Poincaré assumption \eqref{eq:Poincare_V} could then be replaced by the stronger one
	$$| \nabla V(x) | \to \infty \quad \text{when} \quad |x| \to \infty\, .$$
	Some simple families of such potential are given for any $\alpha \in ( 0 , \infty)$
	$$V(x) = \la x  \ra^{1+\alpha} \quad \text{or} \quad V(x) = \la x \ra \log(1+|x|^2)^\alpha \, ,$$
	up to some normalizing additive constant. We refer to \cite{bbcg, villani_2009_hypocoercivity} for improvements and more criteria leading to a Poincaré inequality.
	
	In the following any norm involving only the confinement potential will be denoted by $R_V$.

	\subsection{Interaction potential}
	
	\label{scn:interaction}
	Let us now discuss and state our hypotheses for the interaction potential $k$. We will rather use the corresponding  convolution operator $\KK$ (in the sense of \cite[Definition 2.5.1]{grafakos}) which we assume to be bounded between some Lebesgue spaces. In turn, this yields the existence and uniqueness  \cite[Theorem 2.5.2]{grafakos} of a unique tempered distribution $k$ realizing the convolution. We have
	\begin{equation}\label{eq:convol_op_and_adjoint}
	\KK\rho = k \ast \rho\quad\text{and}\quad \KK^*\rho = \widecheck{k} \ast \rho,
	\end{equation}
	for any Schwartz function $\rho:\R^d\to\R$, where we denoted by $\KK^*$ the formal adjoint of $\KK$.  From there we define the symmetric and skew-symmetric parts of $\KK$, 
	\begin{equation}\label{eq:sym_op_skewsym_op}
	\KK^\text{e}\rho = k^\text{e}\ast \rho\quad\text{and}\quad	\KK^\text{o}\rho = k^\text{o}\ast \rho,
	\end{equation}
	where the superscripts $\text{e}$ and $\text{o}$ refer to the fact that $k^\text{e}$ and $k^\text{o}$ are respectively the even and odd parts of $k$, namely
	\[
	k^\text{e} = \frac{k+\widecheck{k}}{2}\quad\text{and}\quad k^\text{o} = \frac{k-\widecheck{k}}{2}.
	\]	
	Observe that the corresponding Fourier multipliers are related by $\widehat{k^\text{e}} = \Re{\widehat{k}}$ and $\widehat{k^\text{o}} = \di\,\Im{\widehat{k}}$. With obvious notation, we use the superscripts $\alpha = \text{e}, \text{o}$ to denote the part of the potential $\Psi^\alpha_F$ related to the even or odd part of $k$. Before stating our assumptions on $\KK$, let us just mention that by \cite[Lemma 2.5.3]{grafakos}, $\partial^\beta\KK = \KK \partial^\beta$ for any multi-index $\beta\in\N^d$.

	\begin{assu}
	\label{assu:interaction}
	We assume that for some $p, q \in [2, \infty]$ with $q > d$, the following continuity estimates hold:
	\begin{equation}\label{eq:bounds_k}
	\| \KK^\alpha \rho \|_{L^p } + \| \nabla \KK^\alpha \rho \|_{L^q} \le \overline{\kappa}^\alpha \| \rho \|_{L^1 \cap L^2}, \qquad \alpha = e, o \, ,
	\end{equation}
	and we give ourselves, to clarify the dependencies of constants, some arbitrary $\overline{\kappa}_{ \text{max} }>0$ such that
	$$\overline{\kappa} := \overline{\kappa}^\text{e}+ \overline{\kappa}^\text{o}\le \overline{\kappa}_{ \text{max} } \, .$$
	Moreover we assume that there are $\theta\in[0,1]$ and $\underline{\kappa}^\text{e}\ge0$ such that 
	\begin{equation}
	\label{eq:lower_k_sym}
	\langle \KK^\text{e}h , h \rangle \ge - \underline{\kappa}^\text{e}\left(\theta\| h \|_{L^2}^2+(1-\theta)\| h \|_{L^1}^2\right),\text{ for all } h \in L^1 \cap L^2\text{ s.t. }\int h = 0.
	\end{equation}
	Finally we assume the monotonicity property 
	\begin{equation}
	\label{eq:positivity_k}\rho \ge 0 \Rightarrow \KK \rho \ge 0,\text{ for all }\rho \in L^1 \cap L^2.
	\end{equation}
	\end{assu}
	    Let us clarify the upcoming smallness constraints. Assumption~\ref{assu:interaction} involves three constants, and it may be useful to state once and for all how they are used. No smallness whatsoever is required on the boundedness constant $\overline{\kappa}^\text{e}$: this is precisely the point of incorporating $\KK^\text{e}$ into the functional setting, as done in \cite{ADLT}. The constant $\overline{\kappa}_{\text{max}}$ is introduced only as an upper bound so that the assumptions do not appear circular. Only the two remaining constants are assumed to be small, according to the following chain (ignoring $V$ and $\theta$ and considering them to be given universal parameters).
	\begin{itemize}
	\item We assume $\underline{\kappa}^\text{e}<\delta^\text{e}(\overline{\kappa}_{\text{max}})$ for the uniqueness of the steady state (Corollary \ref{cor:uniqueness}) but not for existence or regularity.
	\item We assume $\underline{\kappa}^\text{e}<\delta^\text{e}(C_\star)$ throughout Section~\ref{scn:linearized} so that the twisted norm of Section~\ref{scn:linearized} is equivalent to the $L^2(F_\star)$--norm (Lemma~\ref{lem:Hilbert_space_structure}). Since $C_\star$ depends only on $\overline{\kappa}_{\text{max}}$, this is again a condition of the form $\underline{\kappa}^\text{e}<\delta^\text{e}(\overline{\kappa}_{\text{max}})$.
	\item We assume $\overline{\kappa}^\text{o}<\delta^\text{o}(C_\star,\overline{\kappa}^\text{e},\underline{\kappa}^\text{e},\nu)$, which makes it possible to absorb the non-symmetric contribution $v\cdot\nabla_x\psi^\text{o}_f$ into the dissipation in Lemmas~\ref{lem:Lyapunov_order_0}, \ref{lem:Lyapunov_order_1}, \ref{lem:Lyapunov_regularization_order_1} and~\ref{lem:estimate_twisted_gradx}. Since $C_\star$ depends only on $\overline{\kappa}_{\text{max}}$, since $\overline{\kappa}^\text{e} \le \overline{\kappa}_{\text{max}}$ and since $\underline{\kappa}^\text{e}$ is already small, this amounts to a condition of the form $\overline{\kappa}^\text{o}<\delta^\text{o}(\overline{\kappa}_{\text{max}}, \nu)$.
	\end{itemize}
	    These are exactly the conditions appearing in Theorems~\ref{thm:steady_state} and~\ref{thm:stability} below.
	\begin{rem}\label{rem:coerciv}
	In the bilinear estimate \eqref{eq:lower_k_sym}, $\langle \KK^\text{e}h , h \rangle$ is well-defined as the $L^{p} - L^{p'}$ duality bracket thanks to \eqref{eq:bounds_k}. Observe that it involves only the symmetric part of the kernel since $\langle \KK^\text{e}h , h \rangle = \langle \KK h , h \rangle$. Moreover a sufficient condition for  \eqref{eq:lower_k_sym} to hold is that the negative part of the Fourier transform has some integrability. Indeed one can show that for any $\theta\in[0,1]$,
	\[
	\langle \KK^\text{e}h , h \rangle \ge - C_\theta \|\widehat{k}^\text{e}_-\|_{L^{\frac{1}{1-\theta}}}\|h\|^{2\theta}_{L^2}\ \|h\|^{2(1-\theta)}_{L^1} \ , \quad C_\theta := (2\pi)^{-d(1-\theta)} \ ,
	\]
	by noticing that $\langle \KK^\text{e}h,h\rangle = (2\pi)^{-d}\int_{\R^d} \widehat{k}^\text{e}(\xi)|\widehat h(\xi)|^2\d\xi$, and then using H\"older's inequality followed by an interpolation inequality. This yields \eqref{eq:lower_k_sym} with $\underline{\kappa}^\text{e}=C_\theta\|\widehat{k}^\text{e}_-\|_{L^{\frac{1}{1-\theta}}}$, by Young's inequality. This implies also that if $\widehat{k}^\text{e}_-$ is merely $L^1+L^\infty$ then \eqref{eq:lower_k_sym} holds with $\underline{\kappa}^\text{e}=2\|\widehat{k}^\text{e}_-\|_{L^1+L^\infty}$ and $\theta=1/2$.
	\end{rem}
	
	\begin{rem}\label{rem:positivity}
	The assumption \eqref{eq:positivity_k} could be replaced by 
	\begin{equation*}
	\forall \rho \in L^1 \cap L^2, \quad \rho \ge 0 \Rightarrow \KK \rho \ge -C\|\rho\|_{L^1}.
	\end{equation*}
	Indeed if the kernel $k$ has a negative part in $L^\infty$, then, without loss of generality, it can be translated by a constant, which is $C$ here, to be made positive without changing \eqref{eq:SCVFP}.
	\end{rem}
	
	Let us now give some examples of kernels satisfying, or not, our assumptions.

	\begin{exam}[Lipschitz kernels]
	Given a (possibly non-symmetric) $k\in W^{1,\infty}$, such as \eqref{eq:kSynchrotron}, without loss of generality (see Remark~\ref{rem:positivity}), we can always assume that $k\geq0$. Then, as consequences of Young and H\"older inequalities the hypotheses are satisfied with $p=q=\infty$, and $\underline{\kappa}^\text{e}= \|k^\text{e}\|_{L^\infty}$ and $\overline{\kappa}^\alpha= \| k^\alpha\|_{W^{1,\infty}}$ for $\alpha= \text{e}, \text{o}$.
	\end{exam}

	\begin{exam}[Repulsive Riesz and Coulomb kernel]\label{exam:Riesz_repuls}
	In dimension $d$, let us consider the symmetric repulsive Riesz potential
	\[
	k(x) = \frac{I}{|x|^{d-\alpha}}\,,\quad I\geq0 \, .
	\]
	Then, as a direct consequence of Hardy-Littlewood-Sobolev inequalities, the hypothesis \eqref{eq:bounds_k} is satisfied under the conditions 
	\[
	\frac d2<\alpha \le d\quad\text{and}\quad d\geq2\,,
	\]
	so that $q>d$. This includes the Coulomb kernel ($\alpha=2$)  in dimension $3$. Besides, \eqref{eq:lower_k_sym} and \eqref{eq:positivity_k} follow from the non-negativity of $k$ and $\widehat{k}$. In particular \[\underline{\kappa}^\text{e} = \overline{\kappa}^\text{o} = 0\,\quad\text{and}\quad \overline{\kappa}^\text{e} \underset{I\to0}{=} O(I).\]  Let us insist on the fact that $\underline{\kappa}^\text{e}$ (underlined) is the defect of positivity in the lower bound \eqref{eq:lower_k_sym}, which vanishes here because $\widehat k\ge0$, whereas $\overline{\kappa}^\text{e}$ and $\overline{\kappa}^\text{o}$ (overlined) are the boundedness constants of \eqref{eq:bounds_k}, the kernel being even, $\overline{\kappa}^\text{o}=0$ as well, and $\overline{\kappa}^\text{e}$ is the only non-zero (and possibly large) constant. Observe that a smallness condition on $I$ would translate into a condition on $\overline{\kappa}^\text{e}$ only. As we shall see there will be no such condition for our results to hold, unlike previous works in the literature.
	\end{exam}
	
	\begin{exam}[Singular attractive kernels]
	\label{exam:singular_attractive}
	Singular attractive kernels such as 
	\[
	k(x) = -\frac{I}{|x|^{d-\alpha}}\,,\quad I\geq0,
	\]	
	fail to satisfy hypothesis \eqref{eq:positivity_k}, even if they are translated by a constant (see Remark~\ref{rem:positivity}). This hypothesis is important in our treatment of the steady VFP equation but we discuss an alternative in Section~\ref{sec:particular}. Otherwise, in the light of Remark~\ref{rem:coerciv} and Example \ref{exam:Riesz_repuls}, it is clear that the rest of the assumptions hold with 
	\[\overline{\kappa}^\text{o} = 0,\quad \underline{\kappa}^\text{e} \underset{I\to0}{=} O(I), \quad\text{and}\quad \overline{\kappa}^\text{e} \underset{I\to0}{=} O(I).\]

	\end{exam}

	\begin{exam}[More singular kernels] \label{exam:more_singular}
	When $d \ge 2$, if $1 \le \alpha \le \frac{d}{2}$ (which includes Manev potentials $\alpha=1$, $d=3$, see \cite{choi_2023_global, bdiv}) using Hardy-Littlewood-Sobolev inequalities or Calderon-Zygmund theory only provide $q \in [2, d]$. This situation could be treated with our approach by requiring more regularity on the initial condition, but we decided to restrict this work to $H^1$--type regularity at most. We refer to Section \ref{scn:perspectives} for more on this perspective. When $\alpha < 1$, the potential $\nabla \KK$ maps $L^2$ to some Sobolev space which is less regular, as can be seen by the identification $\KK = c_\alpha (-\Delta)^{- \alpha / 2}$.
	\end{exam}

	\begin{exam}[Other kernels in weak Lebesgue spaces]
	The boundedness assumptions on $\KK$ and $\nabla \KK$ are satisfied whenever $k$ and $\nabla k$ lie in suitable weak Lebesgue spaces  (see \cite{grafakos} for the definition). Namely, if $k \in L^{s, \infty}$ for some $s \in (1, \infty)$ and $\nabla k \in L^{t, \infty}$ for some $t \in \left( \frac{2d}{2+d}, \infty \right)$,  Young's generalized convolution inequality \cite[Theorem 1.4.25]{grafakos} yield $q = \frac{2 t}{2-t} > d$ with
	$$\overline{\kappa}^o = C \left( \| k^o \|_{L^{s, \infty} } + \| \nabla k^o \|_{L^{t, \infty} } \right) \quad \text{and} \quad \overline{\kappa}^e = C \left( \| k^e \|_{L^{s, \infty} } + \| \nabla k^e \|_{L^{t, \infty} } \right) \, ,$$
	for some universal constant $C = C(s, t)$ related to Young's inequality.
	\end{exam}

	\subsection{Main results}
	
	Let us first state our main result concerning steady states of \eqref{eq:SCVFP}. The following theorem is, for the sake of presentation, a restricted version of a more general result, Theorem~\ref{thm:steady} that we prove in Section~\ref{scn:steady_state}.
	\begin{theo}
	\label{thm:steady_state}
	Under Assumptions \ref{assu:confinement} and \ref{assu:interaction} there is $\delta^\text{e}>0$ such that if \[\underline{\kappa}^\text{e}< \delta^\text{e}(\theta, \overline{\kappa}_{ \text{max} }, R_V),\] then the equation \eqref{eq:SCVFP} admits a unique steady state with mass $1$. It is of the form
	\[
	F_\star(x,v) = e^{-V_\star(x)} M(v),
	\]
	with $V_\star-V = \Psi_{F_\star}\in W^{2,\infty}$. The probability measure with density $F_\star$ satisfies the Poincaré and weighted Poincaré inequalities stated in Section~\ref{scn:functional_inequalities_steady_state} for some constant $C_\star$ depending only on $\left\| \left( 1+|\nabla V|^2 \right)e^{-V} \right\|_{L^1 \cap L^\infty}$, the Poincaré constant of $e^{-V} \d x $, as well as $\overline{\kappa}_{ \text{max} }$ and $\theta$.
	\end{theo}
	
	\begin{rem}
	Uniqueness is expected when $\KK$ is symmetric and positive because  the steady state is characterized by minimization of a strictly convex functional. When $\KK$ is non-symmetric and weakly non-positive, a new geometric argument on the free energy is used to prove the uniqueness of steady states (see Section~\ref{scn:uniqueness}). 
	\end{rem}
	
	Our second main result concerns the well-posedness of \eqref{eq:SCVFP}, as well as the decay and regularity of solutions around steady states.	
	
	\begin{theo}
	\label{thm:stability}
	Under Assumptions \ref{assu:confinement} and \ref{assu:interaction}, there are constants $\delta^\text{e}>0$ and $\delta^\text{o}>0$ such that if
	$$\underline{\kappa}^\text{e} < \delta^\text{e}(\theta,\overline{\kappa}_{ \text{max} }, R_V) \qquad \text{and} \qquad \overline{\kappa}^\text{o}< \delta^\text{o}(\theta, \overline{\kappa}_{ \text{max} }, R_V, \nu ),$$ the unique steady state of \eqref{eq:SCVFP} is stable in the following sense. For any $s\in[0,1]$ such that
	\[
	s> s_c := \frac{3}{2} \left( \frac{d}{q} - \frac{1}{3} \right),
	\]
	there are constants $R,C, \lambda>0$ such that the following holds: for any initial datum $F_\ini$ satisfying
	\[\int_{\R^{2d}} F_\ini(x,v) \d x \d v = 1 \quad \text{and} \quad \|F_\ini - F_\star\|_{H^s_xL^2_v(F_\star^{-1})}< R,\]
	the equation \eqref{eq:SCVFP} admits a solution $F\in \mathcal{C}([0,\infty);H^s_xL^2_v(F_\star^{-1}))$ unique among those such that $\nabla_x \Psi_F \in L^2_{t, \mathrm{loc}} L^\infty_x$. Furthermore, there holds
	        
	\begin{align}
	\label{eq:nonlinear_decay}
	\|F(t) - F_\star\|_{H^s_xL^2_v(F_\star^{-1})}&\leq C \|F_\ini - F_\star\|_{H^s_xL^2_v(F_\star^{-1})}e^{-\lambda t},\\
	\label{eq:nonlinear_position_regularization}
	\|F(t) - F_\star\|_{H^1_xL^2_v(F_\star^{-1})}&\leq C \|F_\ini - F_\star\|_{H^s_xL^2_v(F_\star^{-1})}\left(1+t^{-\frac{3}{2}(1-s)}\right)e^{-\lambda t},\\
	\label{eq:nonlinear_velocity_regularization}
	\|F(t) - F_\star\|_{L^2_xH^{1-s}_v(F_\star^{-1})}&\leq C \|F_\ini - F_\star\|_{H^s_xL^2_v(F_\star^{-1})}\left(1+t^{-\frac{1}{2}(1-s)}\right)e^{-\lambda t}.
	\end{align}
	Finally, $F_\ini\mapsto F$ is Lipschitz continuous from $H^s_xL^2_v(F_\star^{-1})$ to $\mathcal{C}([0,\infty);H^s_xL^2_v(F_\star^{-1}))$.
	\end{theo}

	The current work also bring several novelties and improvements on the Cauchy theory for the nonlinear Vlasov-Fokker-Planck equation close to equilibrium.
	\begin{enumerate}
	\item Assumption~\ref{assu:interaction} on the interaction operator $\KK$ encompasses many interactions potentials which are dealt with in a unified manner. They include singular, non-symmetric potentials (see the examples of Section \ref{scn:interaction} and consequences of Theorem~\ref{thm:stability} in the next section).
	\item Concerning the regularity assumption on the initial data, no regularity is required in the $v$ variable, which can be compared with \cite{herau_2016_global, arnold_tosh_VFP}. This is made possible by performing the standard interpolation argument between the linearized theories in $L^2_x L^2_v$ and $H^1_x L^2_v$ (instead of $H^1_{x, v}$), which makes it possible to assume regularity for the initial data in position without assuming any in velocity. Both theories are established by applying the so-called $L^2$--hypocoercivity method of \cite{DMS} to both $F$ and $\nabla_x F$. The price to pay is that the assumption \eqref{eq:smoothness_V} on the confining potential $V$ is required for arbitrarily small $\eps$ (see Remark \ref{rem:hessian_grad_V_eps}).
	\item We prove nonlinear hypoellipticity \eqref{eq:nonlinear_position_regularization} by proving linear hypoellipticity in the presence of a source term in Proposition \ref{prop:H1_linear_hypoellipticity}. During the iterative scheme used to construct the solution, this makes it possible to use the fact that both the linear flow and the approximate solution enjoy regularization properties, thus requiring less regularity for the initial condition.
	\item Theorem~\ref{thm:stability} allows for strongly nonlinear regimes, in the sense that there is no smallness condition on $\overline{\kappa}^\text{e}$. Following the ideas of \cite{ADLT}, at the linearized level, we do not consider the even part of $\KK$, namely $\KK^\text{e}$, to be a ‘‘bad term'' which is to be treated perturbatively. As it appears in the free energy functional \eqref{eq:freeenerg} of the equation, it is in some sense part of the natural metric to study the equation. As a consequence, when this even part is positive, it does not need to be small. In the particular case of Coulomb interactions, we thus prove uniqueness and stability of the equilibria, even in the strongly nonlinear regime (see Example \ref{exam:Riesz_repuls} and Corollary \ref{cor:VPFP}).
	\end{enumerate}
	
	\begin{rem}
	Let us point out that, in \cite{arnold_tosh_VFP} which deals with Coulomb interactions, the main requirement on the confining potential $V$ is that $| \nabla V | e^{-V} \in L^r$ for some $r > d$. In the present work, the integrability assumption $| \nabla V |^2 e^{-V} \in L^r$ for some $r > \frac{d}{2}$ would have been sufficient. Our other assumption on $V$, besides the Poincar\'e inequality, concerns its Hessian and is discussed in Remark~\ref{rem:hessian_grad_V_eps}.
	\end{rem}
	
	\begin{rem}
	\label{rem:hessian_grad_V_eps}
	The assumption \eqref{eq:smoothness_V} with $\eps = 1$ is standard in the context of hypocoercivity with confinement. In this form, it is used in Lemma~\ref{lem:AT_bounded}, so as to get rid of any smallness assumption on $\overline{\kappa}^\text{e}$ (as introduced in \cite{ADLT}), and, through \eqref{eq:poincare_hessianVstar}, in the $H^1_{x,v}$ estimates \eqref{eq:dissipation_grad_x} and \eqref{eq:dissipation_grad_xv}. The version used here (for $\eps \to 0$) can be found in \cite{carrapatoso2021special}, as well as in works of semi-classical analysis (in the stronger form $\nabla^2 V = o( \nabla V)$, see for instance \cite{herau_sc_08, herau_sc_08b, herau_sc_11}). It is used only in the proof of Lemma~\ref{lem:estimate_twisted_gradx}, that is to say in the $H^1_x L^2_v$ theory of Section~\ref{scn:H1xL2v_hypocoercivity}, which makes it possible to consider more spatial regularity without requiring any velocity regularity (as opposed to classical $H^s_{x,v}$--hypocoercivity as in \cite{herau_2016_global}). If \eqref{eq:smoothness_V} only holds for $\eps = 1$, a similar result can be proved, as in \cite{arnold_tosh_VFP}, for initial data in $H^s_{x,v}(F_\star^{-1})$ with $s$ large enough, instead of $H^s_x L^2_v(F_\star^{-1})$.
	\end{rem}
	
	\begin{rem}
	\label{rem:negative_critical_smoothness}
	When $s_c < 0$, that is to say when $q > 3d$, the case $s = 0$ is allowed in Theorem~\ref{thm:stability}. In that case the standard $H^1_{x,v}$ hypoellipticity strategy is enough, so that \eqref{eq:smoothness_V} is only needed with $\eps = 1$, see Remark~\ref{rem:hessian_grad_V_eps}.
	\end{rem}
	
	    \begin{rem}
	        The short-time rates in \eqref{eq:nonlinear_position_regularization}--\eqref{eq:nonlinear_velocity_regularization} are those dictated by the hypoelliptic scaling of the kinetic Fokker--Planck operator, but the target regularity is not optimal. Since we do not consider regularity in position higher than $H^1_x$, these estimates are obtained by interpolating a gain of regularity $L^2_x \to H^1_x$ and a propagation of regularity $H^1_x \to H^1_x$; in particular, the degeneracy of \eqref{eq:nonlinear_velocity_regularization} as $s \to 1$ is an artifact of this interpolation. Higher regularity is the object of a forthcoming work \cite{gervais_herda_higher_order}.
	    \end{rem}
	    
	\subsection{Particular cases and extensions}\label{sec:particular}
	
	The explicit and quantitative stability of equilibria in the case of Poisson interactions was first studied in \cite{herau_2016_global} and improved in \cite{arnold_tosh_VFP}. Our framework makes it possible to reduce even more the regularity required for the initial data $F_\ini$.
	\begin{cor}[Repulsive VPFP]\label{cor:VPFP}
	Under Assumptions \ref{assu:confinement}, consider the Vlasov-Poisson-Fokker-Planck equation with Coulomb interactions, in dimension $d=3$,
	\begin{equation}\label{eq:VPFPc}
	\begin{cases}
		\displaystyle \partial_t F + v \cdot \nabla_x F - \nabla_x \left( \Psi_F + V \right) \cdot \nabla_v F = \nu \nabla_v \cdot \left( v F + \nabla_v F \right),
		\\
		\displaystyle -\delta^2\Delta\Psi_F (t, x) =  \int_{\R^{d} } F(t, x, v)\,  \d v, \\
		\displaystyle F|_{t=0} = F_\ini,
	\end{cases}
	\end{equation}
	for any (arbitrarily small) $\delta>0$. Then for any $\frac14<s\leq 1$, there is a constant $R>0$ such that if
	
	\[\int_{\R^{2d}} F_\ini(x,v) \d x \d v = 1 \quad \text{and} \quad \|F_\ini - F_\star\|_{H^s_xL^2_v(F_\star^{-1})}< R,\]
	then \eqref{eq:VPFPc} has a solution $F\in \mathcal{C}([0,\infty);H^s_xL^2_v(F_\star^{-1}))$, unique among those such that $\nabla_x \Psi_F \in L^2_{t, \mathrm{loc}} L^\infty_x$. Moreover, there are constants $C>0$ and $\lambda>0$ such that for all $t>0$
	\[
	\|F(t) - F_\star\|_{H^s_xL^2_v(F_\star^{-1})}\leq C \|F_\ini - F_\star\|_{H^s_xL^2_v(F_\star^{-1})}e^{-\lambda t}.
	\]
	\end{cor}
	
	In this result, we improve the regularity assumption on the initial condition compared to previous works. We only require $H^{\frac{1}{4}+}_x L^2_v$ instead of (at least) $H^{\frac{1}{2} +}_{x, v}$ regularity \cite{herau_2016_global, arnold_tosh_VFP}. In particular we do not require any regularity in the velocity variable for the initial condition. Moreover the results hold for any arbitrarily small Debye length, namely out of weakly nonlinear regimes. Note that the result requires no smallness condition on $\delta^{-1}$, but it does not assert any uniformity as $\delta\to0$. Indeed, the interaction kernel of \eqref{eq:VPFPc} satisfies Assumption~\ref{assu:interaction} with $\underline{\kappa}^\text{e}=\overline{\kappa}^\text{o}=0$ and $\overline{\kappa}^\text{e}=O(\delta^{-2})$, so that $\overline{\kappa}_{\text{max}}$, the constant $C_\star$ of the functional inequalities satisfied by the steady state, and all the constants derived from it in Sections~\ref{scn:linearized} and~\ref{scn:nonlinear}, in particular $R$, $C$ and $\lambda$, depend on $\delta$ and are not uniform as $\delta\to0$. Quantifying this dependence would require functional inequalities for $\rho_\star$ that are explicit in $\delta$, which we have not attempted. We do not know whether the degeneracy of the decay rate and of the size of the basin of attraction is a genuine feature of the equation or an artifact of the method.

	As pointed out in Example \ref{exam:singular_attractive}, the hypothesis \eqref{eq:positivity_k} does not hold in the case of attractive Newtonian interactions. However, this assumption is necessary only to prove the existence of steady states, which are already known to exist \cite{bouchut_1995_long}. The remaining assumptions still hold, which makes it possible to prove the stability of said equilibria.
	
	\begin{cor}[Attractive VPFP]\label{cor:VPFPn}
	Under Assumptions \ref{assu:confinement}, consider the Vlasov-Poisson-Fokker-Planck equation with Newton interactions, in dimension $d=3$,
	\begin{equation}\label{eq:VPFPn}
	\begin{cases}
		\displaystyle \partial_t F + v \cdot \nabla_x F - \nabla_x \left( \Psi_F + V \right) \cdot \nabla_v F = \nu \nabla_v \cdot \left( v F + \nabla_v F \right),
		\\
		\displaystyle \Delta\Psi_F (t, x) =  \Gamma\int_{\R^{d} } F(t, x, v)\,  \d v, \\
		\displaystyle F|_{t=0} = F_\ini.
	\end{cases}
	\end{equation}
	Then there is $\Gamma_\text{max}>0$ such that if 
	\[
	\Gamma<\Gamma_\text{max},
	\]
	then for any $\frac14<s\leq 1$, there is a constant $R>0$ such that if
	        
	\[\int_{\R^{2d}} F_\ini(x,v) \d x \d v = 1 \quad \text{and} \quad \|F_\ini - F_\star\|_{H^s_xL^2_v(F_\star^{-1})}< R,\]
	then \eqref{eq:VPFPn} has a solution $F\in \mathcal{C}([0,\infty);H^s_xL^2_v(F_\star^{-1}))$, unique among those such that $\nabla_x \Psi_F \in L^2_{t, \mathrm{loc}} L^\infty_x$. Moreover, there are constants $C>0$ and $\lambda>0$ such that for all $t>0$
	\[
	\|F(t) - F_\star\|_{H^s_xL^2_v(F_\star^{-1})}\leq C \|F_\ini - F_\star\|_{H^s_xL^2_v(F_\star^{-1})}e^{-\lambda t}.
	\]
	\end{cor}

	\begin{rem}[Coulomb and Newton kernels in small dimensions] Our approach makes it possible to prove the stability of equilibria in dimension $d=1, 2$ for Coulomb and (small) Newton interactions, up to a minor extra assumption on $V$. Let us only briefly comment on this as these cases have already been investigated in the literature (see \cite{ADLT} and references therein). Our general framework does not exactly fit these situations because the interaction kernel blows up for large relative distances, which prevents $\KK$ from mapping $L^1 \cap L^2$ onto some Lebesgue space. The study of steady states is classical (see \cite{bouchut_1995_long} and references therein) and their smoothness can be proved following the approach of Section \ref{scn:regularity_fixed_point}. This then makes it possible to prove the various functional inequalities of Section \ref{scn:functional_inequalities_steady_state}.	To prove the stability of these equilibria, the assumption that $\KK$ maps in some Lebesgue space is used only in the proof of Lemma \ref{lem:Hilbert_space_structure}, which still holds up to assuming also that (for instance) $\la \cdot \ra^5 \lesssim e^{V/2}$ thanks to the observation that
	$\| \la \cdot \ra^{-3} \KK(\rho) \|_{L^\infty} \le C \| \la \cdot \ra^3 \rho \|_{L^2} \, .$	In the end, the same result as Corollaries \ref{cor:VPFP} and \ref{cor:VPFPn} would then hold with $0 \le s \le 1$. 
	\end{rem}

	In the case of the synchrotron kernel \eqref{eq:kSynchrotron}, our result includes that of  \cite{cesbron2023vlasov}.
	
	\begin{cor}[Synchrotron kernels]\label{cor:synchrotron}
	Under Assumption~\ref{assu:confinement} (for $s = 0$, \eqref{eq:smoothness_V} is only needed with $\eps = 1$, see Remark~\ref{rem:negative_critical_smoothness}), consider the self-consistent Vlasov-Fokker-Planck equation for a relativistic electron bunch in dimension $d=1$
	\begin{equation}\label{eq:synchrotron}
	\begin{cases}
		\displaystyle \partial_t F + v \partial_x F - \partial_x \left( \Psi_F + V \right) \cdot \partial_v F = \nu \partial_v \left( v F + \partial_v F \right),
		\\
		\displaystyle \Psi_F (t, x) =  I \int_{\R\times\R } k(x-y) F(t, y, v)\,  \d v \d y, \\
		\displaystyle F|_{t=0} = F_\ini,
	\end{cases}
	\end{equation}
	where $k\in W^{1,\infty}(\R)$ (typically take $k=k_S$ given by \eqref{eq:kSynchrotron}). Then there is $I_\text{max}>0$ such that if 
	\[
	I < I_\text{max},
	\]
	then for any $0 \le s \le 1$, there is a constant $R>0$ such that if
	\[\int_{\R\times\R} F_\ini(x,v) \d x \d v = 1 \quad \text{and} \quad \|F_\ini - F_\star\|_{H^s_xL^2_v(F_\star^{-1})}< R,\]
	then \eqref{eq:synchrotron} has a solution $F\in \mathcal{C}([0,\infty); H^s_x L^2_v(F_\star^{-1}))$, which is unique in this space. Indeed, since $k$ is Lipschitz, any such solution satisfies $\|\partial_x\Psi_F(t)\|_{L^\infty}\le I\|k'\|_{L^\infty}\|F(t)\|_{L^1}$, so that the uniqueness condition of Theorem~\ref{thm:stability} is automatically fulfilled. Moreover, there are constants $C>0$ and $\lambda>0$ such that for all $t>0$
	\[
	\|F(t) - F_\star\|_{H^s_xL^2_v(F_\star^{-1})}\leq C \|F_\ini - F_\star\|_{H^s_xL^2_v(F_\star^{-1})}e^{-\lambda t}.
	\]
	\end{cor}
	
	\subsection{Outline}
	
	In Section \ref{scn:steady_state} we prove several results related to the steady state equation \eqref{eq:stationary_equation}, namely existence, uniqueness and regularity. These results are of independent interest and hold under weaker assumptions than the general Assumptions \ref{assu:confinement} and \ref{assu:interaction}. We then proceed to prove the functional inequalities related to the measure induced by this steady state in Section \ref{scn:functional_inequalities_steady_state} necessary to the stability analysis.
	
	In Section \ref{scn:linearized} we study the linearized equation around the steady state, using the so called $H^1$--hypocoercivity and hypoellipticity strategies  of \cite{herau_2004_isotropic, villani_2009_hypocoercivity}. We also use the so called $L^2$--hypocoercivity of \cite{DMS}, and more precisely we take inspiration from \cite{ADLT} which incorporates $\KK^e$ into the functional setting so that no smallness assumption on the upper bound $\overline{\kappa}^e$ is required. In turn we consider $\KK$ to be almost positive symmetric, that is we assume $\overline{\kappa}^o$ and $\underline{\kappa}^e$ small. One novelty of this work is that we apply the $L^2$--hypocoercivity strategy to $\nabla_x (F - F_\star)$ in order to build a linear theory in $H^1_x L^2_v$ so that no regularity in velocity is required, which requires the mild growth assumption \eqref{eq:smoothness_V} on $V$ for arbitrarily small $\eps$.
	
	In Section \ref{scn:nonlinear} we combine all previous results to prove Theorem \ref{thm:stability}. We first interpolate the hypocoercivity and hypoellipticity results to deduce that the linearized flow maps initial data and source continuously for some appropriate ‘‘natural'' norm. In Section \ref{scn:nonlinear_estimate}, we show that the nonlinearity is bounded in this ‘‘natural'' norm, so that, in Section \ref{scn:nonlinear_proof}, the proof of Theorem \ref{thm:stability} reduces to a straightforward application of Banach-Picard's fixed point theorem.
	
	In Section \ref{scn:perspectives}, we present some perspectives, namely some possible improvements regarding the regularity assumptions on $\nabla \KK$ and the initial data, and some continuation of the present work.

	\section{Existence, uniqueness and regularity of the steady state}
	
	\label{scn:steady_state}
	
	The first goal of this section is to prove Theorem~\ref{thm:steady} below (and thus Theorem~\ref{thm:steady_state}) concerning the existence, uniqueness and regularity of steady states. From this construction, we will derive various functional inequalities related to the steady state in Section~\ref{scn:functional_inequalities_steady_state}.

	\begin{theo}\label{thm:steady} Assume that the confinement potential $V$ satisfies, for some $N\geq2$
	\[
	\begin{cases}
	\|e^{-V}\|_{L^1} = 1\\
	\forall 1\le n \le N,\quad \left(1+| \nabla^n V |^\frac{N}{n}\right) e^{-V}\in L^1 \cap L^\infty
	\end{cases}
	\]
	and that the interaction operator $\KK$ satisfies that for some $p,q\in[2,\infty]$,
	\[
	\begin{cases}
	\| \KK \|_{L^1 \cap L^2 \to L^p}+\| \nabla \KK \|_{L^1 \cap L^2 \to L^q} \leq  \overline{\kappa},\\
	\forall \rho \in L^1 \cap L^2, \quad \rho \ge 0 \Rightarrow \KK \rho \ge 0,\\
	\|\KK^* \left( e^{-V} \right)\|_{ L^\infty}\leq\zeta.
	\end{cases}
	\]
	Then, smooth and rapidly decaying solutions to the stationary Vlasov-Fokker-Planck equation 
	\begin{equation}\label{eq:steadySCVFP}
	v \cdot \nabla_x F_\star - \nabla_x \left( \Psi_{F_\star} + V \right) \cdot \nabla_v F_\star = \nu \nabla_v \cdot \left( v F_\star + \nabla_v F_\star \right)
	\end{equation}
	are given by all functions
	\[
	F_\star(x,v) = \rho_\star(x) M(v),
	\]
	with $\rho_\star$ a solution of the fixed point problem
	\begin{equation}
	\label{eq:stationary_equation}
	\rho_\star = \TT(\rho_\star) := \frac{\SS(\rho_\star)}{ \| \SS(\rho_\star) \|_{L^1} }, \qquad \SS(\rho) := e^{-V-\KK \rho} \,.
	\end{equation}
	Under the hypotheses on the potentials, there exists a solution $\rho_\star = e^{-V_\star}>0$ to \eqref{eq:stationary_equation}. Moreover, if additionally  $q\geq d/2$ and $N>1+d/q$, then any solution satisfies $V_\star-V\in W^{2,\infty}$. Finally under the additional assumptions that \eqref{eq:lower_k_sym} holds with $\underline{\kappa}^\text{e}<\delta^\text{e}(\theta, \overline{\kappa}, R_V)$, then $\rho_\star$ is uniquely defined, and therefore there is a unique steady state to \eqref{eq:steadySCVFP}.
	\end{theo}
	
	\subsection{Existence of fixed points}
	
	From here we consider that the domain of $\TT$ is the convex set
	$$\CC_M := \left\{ \rho \in L^1 \cap L^r \, : \, \| \rho \|_{L^1} \le 1, ~ \| \rho \|_{L^r} \le M, ~ \rho\geq0 \right\} \, $$
	with $M>0$ to be determined. The purpose of the Lebesgue index $r$ is to make the following estimates more general than what is strictly needed to prove Theorem~\ref{thm:steady}. In the proof of Theorem~\ref{thm:steady}, we shall take $r=2$.
	
	\begin{lem}[Properties of $\TT$]
	\label{lem:properties_T}
	Assume that the interaction potential satisfies the bound
	\begin{equation*}
	\| \nabla \KK \|_{L^1 \cap L^r \to L^q} \le \overline{\kappa}, \quad 1 \le r \le q \le \infty \, ,
	\end{equation*}
	and it is assumed to be positive in the sense that
	\begin{equation*}
	\forall \rho \in L^1 \cap L^r, \quad \rho \ge 0 \Rightarrow \KK \rho \ge 0 \, .
	\end{equation*}
	The confining potential is assumed to be such that
	\begin{equation}
	\label{eq:weighted_Sobolev_V}
	\|e^{-V}\|_{L^1} = 1\, \quad \text{and}\quad \left\| (1+|\nabla V|)\, e^{-V} \right\|_{L^1 \cap L^\infty} \le R_V.
	\end{equation}
	Finally, we assume the following relationship between $\KK$ and $V$:
	\begin{equation}
	\label{eq:K_adj_confinement}
	\left\| \KK^* \left( e^{-V} \right)\right\|_{ L^{\infty} } \le \zeta \, .
	\end{equation}
	Under these assumptions, $\TT$ satisfies for any $\rho \in \CC_M$
	\begin{equation}
	\label{eq:upper_bound_T}
	\left\| \TT(\rho) \right\|_{L^s} \le e^{\zeta} R_V \quad \text{and} \quad \quad \| \TT(\rho) \|_{L^1} = 1 \,,\quad s\in[1,\infty],
	\end{equation}
	and more precisely the concentration estimate
	\begin{equation}
	\label{eq:concentration_T}
	\left\| e^{  V / s' } \TT(\rho) \right\|_{L^s} \le e^{\zeta/s'}\, {,\quad s\in[1,\infty],}
	\end{equation}
	as well as the regularity estimate
	\begin{equation}
	\label{eq:upper_bound_grad_T}
	\left\| \nabla \TT(\rho) \right\|_{L^s} \le R_V e^{\zeta} \left( 1 + \overline{\kappa} + \overline{\kappa} M \right)\,{,\quad s\in[1,q]}.
	\end{equation}
	Furthermore, $\TT(\CC_M)$ is {precompact in }$L^1$, and $\TT$ satisfies the continuity estimate
	\begin{equation}
	\label{eq:Holder_T}
	\| \TT(\rho) - \TT(\sigma) \|_{L^s} \le 2e^{\zeta}  \zeta^{ 1 / s } R_V^{1/s'} \| \rho-\sigma \|_{L^1}^{1/s} \,,\quad {s \in [1, \infty]} .
	\end{equation}
	\end{lem}
	
	\begin{proof}
	We first prove the integrability bound \eqref{eq:upper_bound_T}, then the gradient estimate \eqref{eq:upper_bound_grad_T}. We then deduce the compactness from these bounds, and finally proceed to prove the Hölder continuity in $L^s$ (where $s < \infty$).
	
	\medskip
	\step{1}{Integrability estimate}
	We start by exhibiting a lower bound on the mass of $\SS(\rho)$, using Jensen's inequality with respect to the measure $e^{-V(x)} \, \d x$:
	\[
	\int_{\R^d} \SS(\rho) \, \d x  =  \int_{\R^d} e^{- V - \KK \rho} \, \d x  \ge \exp\left( - \int_{\R^d} \KK \rho e^{-V} \, \d x \right)  = \exp\left( - \int_{\R^d} \rho \,  \KK^*  \left( e^{-V} \right) \d x \right) \, ,
	\]
	from which we deduce using \eqref{eq:K_adj_confinement} and the fact that $\| \rho \|_{L^1} \le 1$
	\begin{equation}
	\label{eq:lower_bound_mass_S}
	\int_{\R^d} \SS(\rho) \, \d x \ge e^{ - \zeta } \, .
	\end{equation}
	Clearly $\left\| \TT(\rho) \right\|_{L^1} = 1$ and since $\KK \rho \ge 0$, one easily deduces
	$$\left\| e^V \TT(\rho) \right\|_{L^\infty} \le e^{\zeta}$$
	which then yields the upper bound \eqref{eq:upper_bound_T}  and the concentration estimate \eqref{eq:concentration_T}.
	
	\step{2}{Regularity estimate}
	Starting from the observation (since $\KK \rho \ge 0$)
	$$| \nabla \SS(\rho) | = | \nabla V + \nabla \KK(\rho) | \SS(\rho) \le | \nabla V + \nabla \KK(\rho) | e^{-V}$$
	and denoting $\frac{1}{s} = \frac{1}{q} + \frac{1}{t}$, one obtains the  bound
	\begin{align*}
	\| \nabla \SS(\rho) \|_{L^s} & \le \left\| \nabla V e^{-V} \right\|_{L^s} + \| \nabla \KK(\rho) \|_{L^q} \left\| e^{-V} \right\|_{L^t} \\
	& \le R_V (1 + \overline{\kappa} \| \rho \|_{L^1 \cap L^r}) \\
	& \le R_V (1 + \overline{\kappa} + \overline{\kappa} M ) \, ,
	\end{align*}
	which, combined with \eqref{eq:lower_bound_mass_S} yields \eqref{eq:upper_bound_grad_T}.
	
	\step{3}{Compactness}
	Let us consider some sequence $( \rho_n) \subset \CC_M$ and denote its image $\sigma_n := \TT(\rho_n)$. Since \eqref{eq:concentration_T} and \eqref{eq:upper_bound_grad_T} imply that for some $C > 0$
	$$\left\| e^V \sigma_n \right\|_{L^\infty} + \| \nabla \sigma_n \|_{L^s} \le C \, ,$$
	the Rellich-Kondrachov theorem states that some subsequence, which we still denote $\sigma_n$, converges in $L^1_{\text{loc}}$ to some $\sigma_\infty$ satisfying the same bounds as $\sigma_n$. The global convergence follows from the weighted $L^\infty$ bound because for any $R > 0$
	$$\| \sigma_n - \sigma_\infty \|_{L^1} \le 2C \left\| \mathbf{1}_{|x| \ge R} e^{-V} \right\|_{L^1} + \| \mathbf{1}_{|x| \le R} (\sigma_n - \sigma_\infty) \|_{L^1}$$
	and $e^{-V} \in L^s$.  The right-hand side can be made as small as desired by taking $R$ large and then $n$ large.
	
	\step{4}{Continuity estimates}
	Using the inequality
	$$\forall a, b \ge 0 , \quad \left|e^{-a} - e^{-b}\right| \le |a-b| $$
	with $a = \KK(\rho) \ge 0$ and $b = \KK (\sigma) \ge 0$, as well as $| \KK(\rho-\sigma)| \le \KK(|\rho-\sigma|)$, which are consequences of the non-negativity of $\KK$, one has the control
	$$\left| \SS(\rho) - \SS(\sigma) \right| \le \KK (|\rho - \sigma|) e^{-V}  \, .$$
	We then deduce that
	\begin{align}
	\notag
	\| \SS(\rho) - \SS(\sigma) \|_{L^1} & \le \left\| (\rho - \sigma) \KK^* \left( e^{-V} \right) \right\|_{L^1} \\
	\notag
	& \le \| \rho - \sigma \|_{L^1} \left\| \KK^* \left( e^{-V} \right) \right\|_{L^\infty } \\
	\label{eq:Lipschitz_S}
	& \le \zeta \| \rho - \sigma \|_{L^1} \, ,
	\end{align}
	and observing that the following identity holds
	$$\TT(\rho) - \TT(\sigma) = \frac{\SS(\rho) - \SS(\sigma)}{\| \SS(\rho) \|_{L^1}} + \TT(\sigma) \frac{\| \SS(\rho) \|_{L^1} - \| \SS(\sigma) \|_{L^1} }{ \| \SS(\rho) \|_{L^1}  } \, ,$$
	one deduces from \eqref{eq:lower_bound_mass_S} and \eqref{eq:Lipschitz_S}
	$$\| \TT(\rho) - \TT(\sigma) \|_{L^1} \le 2 \zeta e^{\zeta} \| \rho - \sigma \|_{L^1}  \ .$$
	Interpolating with the bound
	$$\| \TT(\rho)-\TT(\sigma) \|_{L^\infty} \le 2 e^{\zeta} R_V \, ,$$
	one deduces \eqref{eq:Holder_T}. This concludes the proof.
	\end{proof}
	
	The following existence result is a direct consequence of Lemma \ref{lem:properties_T} and an application of Schauder's theorem with $M = e^{\zeta} R_V$ and endowing $\CC_M$ with the topology of $L^1$ (note that $\TT$ is indeed Lipschitz continuous by virtue of \eqref{eq:Holder_T}).
	
	\begin{prop}[Existence of the steady state]
	Under the assumptions of Lemma \ref{lem:properties_T}, the mapping $\TT:L^1\to L^1$ admits a fixed point in $\mathcal{C}_M$ with $M = e^{\zeta} R_V$. Furthermore, any other fixed point satisfies the estimates of Lemma~\ref{lem:properties_T}.
	\end{prop}
	
	\subsection{Regularity of fixed points}	
	
	\label{scn:regularity_fixed_point}
	
	\begin{lem}[Regularity of the steady state]
	\label{lem:regularity_steady_state}
	Under the assumptions of Lemma \ref{lem:properties_T}, and assuming furthermore $q \ge \frac{d}{2}$, as well as for some $N \ge 2$
	$$\forall n \le N, \quad \left\| | \nabla^n V |^{\frac{N}{n}} e^{-V} \right\|_{L^1 \cap L^\infty} \le R_V \, ,$$
	there exists some $C=C(N, \overline{\kappa}, \zeta, R_V)$ such that any fixed point $\rho_\star$ of $\TT$ satisfies the inductive estimate, for $1\leq n \leq N$,
	\begin{equation}
	\label{eq:inductive_Sobolev}
	\begin{split}
		\| \rho_\star \|_{W^{n, 1} } + \| \rho_\star \|_{W^{n, q} } \le & C \left( 1 + \| \rho_\star \|_{W^{n-1, 1} } +  \| \rho_\star \|_{W^{n-1, r} } \right)^n \, .
	\end{split}
	\end{equation}
	\end{lem}
	
	\begin{proof}
	Let us consider $\rho_\star$ of the form $\rho_\star = e^{-V_\star}$. We first prove a Fa\`a di Bruno--type inequality, and then proceed to prove \eqref{eq:inductive_Sobolev}.
	
	\step{1}{A Fa\`a di Bruno--type inequality}
	Let us prove that there exists some $C = C(N)$ such that for any $n\in\{1,\dots,N\}$,  one has
	\begin{equation}
	\label{eq:FaaDiBruno}
	\left| \nabla_x^n \rho_\star \right| \le C \, \rho_\star \sum_{k=1}^n \left| \nabla_x^k V_\star \right|^{\frac nk} .
	\end{equation}
	This is true for $n = 1$, let us prove it by induction. Assume it holds for some $n \ge 1$ and consider $\alpha \in \N^d$ such that $| \alpha | = n$, there holds by virtue of Leibniz's formula for partial derivatives
	\begin{align*}
	\partial^{\alpha + \mathbf{e}_j } \rho_\star & = - \partial^{\alpha}(\partial^{\mathbf{e}_j }V_\star \rho_\star) = - \sum_{\beta \le \alpha} \binom{\alpha}{\beta} \left( \partial^{\beta + \mathbf{e}_j } V_\star \right) \left( \partial^{\alpha - \beta} \rho_\star \right)
	\end{align*}
	and thus, letting $k=|\beta| \in [0, n]$, there holds for some $C = C(N) > 0$
	\begin{align*}
	\left| \partial^{\alpha + \mathbf{e}_j } \rho_\star \right|
	& \le C \, \sum_{ k = 1 }^{n+1} \left( \rho_\star^{\frac{k}{n+1}} \left| \nabla^k V_\star \right| \right) \left( \rho_\star^{-\frac{k}{n+1}} \left| \nabla^{n+1-k} \rho_\star \right| \right) \, .
	\end{align*}
	Using Young's inequality with the exponents $\frac{k}{n+1} + \frac{n+1-k}{n+1} = 1$, and then re-indexing, one obtains, up to enlarging $C$
	\begin{align*}
	\left| \partial^{\alpha + \mathbf{e}_j } \rho_\star \right|
	& \le  C \, \rho_\star \left(\sum_{ k = 1 }^{n+1}  \left| \nabla^k V_\star \right|^{ \frac{n+1}{k} } + \sum_{ k = 1 }^{n}\left( \rho_\star^{-1} \left| \nabla^{k} \rho_\star \right| \right)^{ \frac{n+1}{k} } \right)
	\end{align*}
	Using the inductive assumption and Minkowski's inequality, we deduce
	\begin{align*}
	\left| \partial^{\alpha + \mathbf{e}_j } \rho_\star \right|
	& \le  C \, \rho_\star \sum_{ k = 1 }^{n+1} \left| \nabla^k V_\star \right|^{ \frac{n+1}{k} }
	\end{align*}
	Summing over $j = 1, \dots, d$, we deduce \eqref{eq:FaaDiBruno} at rank $n+1$.
	
	\step{2}{Inductive control}
	Since $\nabla V_\star = \nabla V + \nabla \KK (\rho_\star)$, the previous step yields the estimate
	\begin{align*}
	\| \nabla^n \rho_\star \|_{L^p} \le & C \sum_{k = 1}^n \left\| \rho_\star | \nabla^k V |^{\frac{n}{k}} \right\|_{L^p} + C \sum_{k = 0}^{n-1} \left\| \rho_\star | \nabla^k ( \nabla \KK (\rho_\star) ) |^{\frac{n}{k+1}} \right\|_{L^p} \, .
	\end{align*}
	The first term is bounded by $C = C(\zeta, R_V, N)$ because of \eqref{eq:concentration_T}, so we focus on the inductive term. Hölder's inequality yield
	\begin{align*}
	\left\| \rho_\star | \nabla^k ( \nabla \KK (\rho_\star)) |^{\frac{n}{k+1}} \right\|_{L^1 \cap L^q} \le & \| \rho_\star \|_{ L^{q'} \cap L^\infty } \left\| |\nabla^k ( \nabla \KK (\rho_\star)) |^{ \frac{n}{k+1} } \right\|_{L^q } \\
	= & {\| \rho_\star \|_{ L^{q'} \cap L^\infty } \left\|\nabla^k\nabla \KK (\rho_\star) \right\|_{L^{\frac{n q}{k+1} } }^{\frac{n}{k+1}}}
	\end{align*}
	One then interpolates for $k = 0, \dots, n-1$ each term between the end points $k=0$ and $k=n-1$ with $\theta = \frac{k}{n-1}$ using Gagliardo-Nirenberg's inequality :
	$${\left\|\nabla^k\nabla \KK (\rho_\star) \right\|_{L^{\frac{n q}{k+1} } }^{\frac{n}{k+1}}} \le C \left\| \nabla \KK ( \rho_\star )\right\|_{ L^{ nq } }^{1-\theta} {\left\| \nabla^{n-1}\nabla \KK (\rho_\star) \right\|_{ {L}^{q } }^{\theta}} \, ,$$
	and then, since $N \ge 2$ and $q \ge \frac{d}{2}$, the first term is controlled using the embedding $W^{n-1, q} \subset L^{qn}$:
	$$\left\| \nabla \KK  (\rho_\star) \right\|_{ {W}^{k,  \frac{nq}{k+1} } } \le C \left\| \nabla \KK (\rho_\star) \right\|_{ {W}^{n-1,  q } }.$$
	In conclusion, using \eqref{eq:upper_bound_T}, there holds
	$$\| \nabla^n \rho_\star \|_{L^1 \cap L^q} \le C + C \sum_{ k = 0 }^{n-1} \| \nabla \KK (\rho_\star) \|_{W^{n-1, q} }^{\frac{n}{k+1}} \, ,$$
	from which we deduce \eqref{eq:inductive_Sobolev} using the boundedness of $\nabla \KK$, this concludes the proof.		
	\end{proof}
	
	\begin{prop}[Regularity of the confinement perturbation]
	\label{prop:regularity_confinement_perturbation}
	Under the assumptions of Lemma \ref{lem:regularity_steady_state}, and assuming that $N > 1 + \frac{d}{q}$ and
	$$\| \KK \|_{L^1 \cap L^r \to L^p} \le \overline{\kappa} \quad \text{for some} \quad p \in [1, \infty] \, ,$$
	any fixed point $\rho_\star = e^{-V_\star}$ of $\TT$ satisfies for some $C = C (N, \overline{\kappa}, \zeta, R_V)$
	$$\| \rho_\star \|_{W^{N, 1} } + \| \rho_\star \|_{W^{N, q} } + \| V - V_\star \|_{W^{2, \infty} } \le C \, .$$
	\end{prop}
	
	\begin{proof}
	As a direct consequence of Lemma \ref{lem:regularity_steady_state}, there is some $C_0=C_0(N, \overline{\kappa}, \zeta, R_V)$ such that
	$$\| \rho_\star \|_{W^{N, 1} } + \| \rho_\star \|_{W^{N, q} } \le C_0$$
	and therefore, up to enlarging $C_0$, and noticing that $\nabla \KK(\rho_\star) = \nabla (V_\star - V)$
	$$\| \nabla (V_\star - V) \|_{W^{N, q}} \le C_0 \, ,$$
	and then by Sobolev embedding
	$$\| \nabla (V_\star - V) \|_{W^{1, \infty}} \le C_0 \, .$$
	Finally, noticing that for any smooth enough $u$ and $x \in \R^d$, there holds
	\begin{align*}
	|u(x)| & \le \frac{1}{\omega_d} \int_{ |x-y| \le 1 } |u(x)-u(y)| \, \d y + \frac{1}{\omega_d} \int_{ |x-y| \le 1 } |u(y)| \, \d y \\
	& \le \| \nabla u \|_{L^\infty} + \frac{1}{\omega_d^{1 / p}} \| u \|_{L^p} \, ,
	\end{align*}
	where $\omega_d$ is the volume of the unit ball, taking $u = V_\star - V = \KK(\rho_\star)$, this estimate guarantees that $V_\star-V \in L^\infty$. This concludes the proof.
	\end{proof}

	\subsection{Uniqueness of fixed points}	\label{scn:uniqueness}	
	Let us introduce the macroscopic free energy functional (see \eqref{eq:freeenerg} for its kinetic counterpart)
	$$\forall \rho \in X, \quad \FF[\rho] := \int_{\R^d} \left( V + \frac{1}{2} \KK^\text{e}(\rho) \right) \rho \, \d x + \int_{\R^d} \rho \log \rho \d x $$
	defined in the convex set
	$$X := \left\{ \rho \in L^1 \cap L^2 \, : \, 0<\rho<R_Ve^\zeta, \, \int \rho \, \d x = 1 \right\}  \,,$$		
	where the upper bound is that of \eqref{eq:upper_bound_T} with $s=\infty$. Here we assume that $\KK^\text{e}:L^1 \cap L^2\to L^p$ with $p\in[2,\infty]$. Observe that, compared to the previous section (see Proposition~\ref{prop:regularity_confinement_perturbation}), the range of indices for $p$ has been shrunk and we take $r=2$. We also introduce the
	differential of $\FF$ with respect to $\rho$:
	$$\forall h \in Y, \quad \d_\rho \FF[\rho] \cdot h = \int_{\R^d} \left( V + \KK^\text{e}(\rho) + \log \rho \right) h \, \d x \, ,$$ 
	where we denoted the vector space
	$$Y = \left\{ h \in L^1 \cap L^2 \, : \, \int_{\R^d} h \, \d x = 0 \right\} \, .$$
	\begin{prop}[Uniqueness of the steady state]
	Under the assumptions of Lemma \ref{lem:properties_T}, the following properties hold.
	\begin{itemize}
		\item[\rm a)] If $\KK^\text{e}$ is such that $\FF$ is strictly convex, then $\TT$ has a unique fixed point.		 
		\item[\rm b)] If $\mathcal{K}^\text{o}= 0$, fixed points of $\TT$ coincide exactly with critical points of $\FF$ with mass $1$. In particular, under the hypothesis of {\rm a)}, the unique fixed point of $\TT$ is the unique minimizer of $\FF$ under the constraint of unit mass.
	\end{itemize} 
	\end{prop}
	\begin{proof}
	One observes that any $\rho \in X$ of the form
	$$\rho = c \,  e^{-V - \KK \rho} \quad \text{for some} \quad c \in (0, \infty)$$
	is such that the first order differential of $\FF$ writes
	$$\forall h \in Y, \quad \d_\rho \FF[\rho] . h = - \int_{\R^d} \KK^\text{o}(\rho) h \, \d x \, ,$$
	and in particular, the skew-symmetry of $\KK^\text{o}$ implies
	\begin{equation}
		\label{eq:free_energy_fixed_point_identity}
		(\d_\rho \FF[\rho_1] - \d_\rho \FF[\rho_0]). (\rho_1 - \rho_0) = 0 \, .
	\end{equation}
	for any $\rho_0$ and $\rho_1$ of the previous form. Therefore, by strict convexity of $\FF$, $\rho_1= \rho_0$ which proves a). 		The property b) is an immediate consequence of the expressions of $\d_\rho \FF$ and $\TT$.
	\end{proof}
	
		\begin{cor}\label{cor:uniqueness} Under the assumptions of Proposition \ref{prop:regularity_confinement_perturbation} and if there is $\underline{\kappa}^\text{e}<\delta^\text{e}(\theta, R_V, \zeta)$ such that 
	$$\langle \KK^\text{e}h , h \rangle \ge - \underline{\kappa}^\text{e}\left(\theta\| h \|_{L^2}^2+(1-\theta)\| h \|_{L^1}^2\right),\text{ for all } h \in L^1 \cap L^2\text{ s.t. }\int h = 0.$$
	then the mapping $\TT$ admits a unique fixed point. 
	\end{cor}
	
	\begin{proof}
	Let us show that under this additional hypothesis, $\FF$ is strictly convex. Indeed, one has
	\[(\d_\rho \FF[\rho_1] - \d_\rho \FF[\rho_0]). (\rho_1 - \rho_0) 
	= \la \KK^\text{e} \rho_1-\rho_0 , \rho_1-\rho_0 \ra_{L^2}  + \int (\rho_1-\rho_0)\log\left(\frac{\rho_1}{\rho_0}\right). \]
	On the one hand by the Csiszár-Kullback-Pinsker inequality, one has 
	\[
	\int (\rho_1-\rho_0)\log\left(\frac{\rho_1}{\rho_0}\right)\geq \|\rho_1-\rho_0\|_{L^1}^2
	\]			
	On the other hand, thanks to the upper bound on functions in $X$, one also has
	\[
	\int (\rho_1-\rho_0)\log\left(\frac{\rho_1}{\rho_0}\right) \geq \frac{1}{R_V e^\zeta}\|\rho_1-\rho_0\|_{L^2}^2
	\]
	where we used the identity
	$$\forall x, y > 0, \quad (x - y) \log \left( \frac{x}{y} \right) = (x-y)^2 \int_0^1 \frac{1}{ t y + (1-t) x } \d t \, . $$
	By combining these controls of the $L^1$ and $L^2$ distances with the assumption, for a small enough $\delta^\text{e}(\theta, R_V, \zeta)$, one shows
	\[
	(\d_\rho \FF[\rho_1] - \d_\rho \FF[\rho_0]). (\rho_1 - \rho_0)\geq0
	\]
	with equality only if $\rho_0=\rho_1$.
	\end{proof}
	\begin{rem}
	If $\theta = 0$ then clearly $\delta^\text{e}=1$.
	\end{rem}
	\begin{rem}
	The lower bound assumption is not a technical one, when the interaction kernel has large negative modes one may observe phase transition and appearance of several equilibria. We refer to \cite{carrillo_2020_long} in the case of the McKean-Vlasov equation on the torus without confinement.
	\end{rem}
	
	\subsection{Functional inequalities related to the steady state}
	
	\label{scn:functional_inequalities_steady_state}
	
	These properties are an immediate consequence of Lemma \ref{lem:regularity_steady_state} and the Holley-Stroock perturbation principle, and we just point out that \eqref{eq:Poincare_gV_rho_star} and \eqref{eq:Poincare_gV2_rho_star} follow from an integration by part argument.
	
	\begin{lem}[Uniform bounds]
	\label{lem:uniform_bounds_rho_star}
	Under the assumptions of Proposition \ref{prop:regularity_confinement_perturbation} and assuming that for any $\eps > 0$ there exists some $C_\eps > 0$ such that
	\begin{equation*}
	\left| \nabla^2 V \right| \le \eps | \nabla V | + C_{\eps} \, ,
	\end{equation*}
	then any fixed point $\rho_\star$ of $\TT$ is such that, for any $\eps > 0$, there exists a constant $C_{\star, \eps} = C_{\star, \eps} (N, \overline{\kappa}, \zeta, R_V, C_\eps) > 0$ such that
	\begin{equation}
	\label{eq:hess_grad_V_star}
	\left| \nabla^2 V_\star \right| \le \eps | \nabla V_\star | + C_{\star,\eps} \, .
	\end{equation}
	Furthermore, there exists some $C_\star > 0$ such that
	\begin{equation}
	\label{eq:moment_rho_star}
	\left\| \left( 1 + | \nabla V_\star|^2 \right) \rho_\star \right\|_{L^1 \cap L^\infty} \le C_\star \, .
	\end{equation}
	\end{lem}
	
	\begin{lem}[Poincaré inequalities]
	\label{lem:poincares}
	Under the assumptions of Lemma \ref{lem:uniform_bounds_rho_star} and assuming that for any
	$$u \in H^1\left( e^{-V} \right) \quad \textnormal{such that} \quad \int_{\R^d} u(x) e^{-V(x)} \d x = 0$$
	the Poincaré inequality holds:
	\begin{equation*}
	\| u \|_{L^2\left( e^{-V} \right)} \le C_P \| \nabla u \|_{L^2\left( e^{-V} \right) } \, ,
	\end{equation*}
	then, for any fixed point $\rho_\star = e^{-V_\star}$ of $\TT$, there exists some constant $C_\star = C_\star (N, \overline{\kappa}, \zeta, R_V, C_P) > 0$ satisfying for any
	$$u \in H^1(\rho_\star) \quad \textnormal{such that} \quad \int_{\R^d} u(x) \rho_\star(x) \d x = 0$$
	the Poincaré inequality:
	\begin{equation}
	\label{eq:Poincare_rho_star}
	\| u \|_{L^2(\rho_\star)} \le C_\star \| \nabla u \|_{L^2(\rho_\star)} \, ,
	\end{equation}
	as well as the following weighted variations:
	\begin{equation}
	\label{eq:Poincare_gV_rho_star}
	\| u \, |\nabla V_\star| \|_{L^2(\rho_\star)} \le C_\star \| \nabla u \|_{L^2(\rho_\star)} \, ,
	\end{equation}
	\begin{equation}
	\label{eq:Poincare_gV2_rho_star}
	\| u \, |\nabla V_\star|^2 \|_{L^2(\rho_\star)} \le C_\star \left( \| \nabla u \|_{L^2(\rho_\star)} + \left\| \nabla u \, |\nabla V_\star | \right\|_{L^2(\rho_\star)} \right) \, .
	\end{equation}
	\end{lem}
	
	\begin{lem}[Weighted Sobolev spaces comparison]
	Under the assumptions of Lemma \ref{lem:poincares} and for any fixed point $\rho_\star = e^{-V_\star}$ of $\TT$, there exists some $C=C(C_\star)$ such that
	\begin{equation}
	\label{eq:weighted_sobolev_space_comparison}
	\frac{1}{C} \| u \, \rho_\star^{1/2} \|_{H^s} \le \| u \|_{H^s(\rho_\star)}  \le C \| u \|_{H^s} \, , \quad s \in [0, 1] \, ,
	\end{equation}
	furthermore, denoting $\nabla^* := -\nabla + \nabla V_\star$, there also holds for $s = 1$
	\begin{equation}
	\label{eq:twisted_spatial_H1_equivalence}
	\frac{1}{C} \| u \|_{H^1(\rho_\star)} \le \| u \|_{L^2(\rho_\star)} + \| \nabla^* u \|_{L^2(\rho_\star)} \le C \| u \|_{H^1(\rho_\star)} \, .
	\end{equation}
	\end{lem}
	
	\begin{proof}
	The bounds of \eqref{eq:weighted_sobolev_space_comparison} follow from \eqref{eq:moment_rho_star} and \eqref{eq:Poincare_gV_rho_star} for $s=0, 1$, which then holds for $s \in [0, 1]$ by interpolation. To prove \eqref{eq:twisted_spatial_H1_equivalence}, one writes
	\[
	\| \nabla u \|^2_{L^2(\rho_\star)}  = \la \nabla u , \nabla u \ra_{L^2(\rho_\star)} = \la \nabla^* u , \nabla^* u \ra_{L^2(\rho_\star)} + \left\la \left[ \nabla^* , \nabla \right] u , u \right\ra_{L^2(\rho_\star)} \, .
	\]
	Observing that $[\nabla, \nabla^* ] = \Delta V_\star$, we deduce from \eqref{eq:hess_grad_V_star} and \eqref{eq:Poincare_gV_rho_star} that the second term is in fact a lower order term:
	$$\left| \left\la \left[ \nabla^* , \nabla \right] u , u \right\ra_{L^2(\rho_\star)} \right| \le \frac{1}{2} \| \nabla u \|_{L^2(\rho_\star)}^2 + C \|u\|^2_{L^2(\rho_\star)}$$
	and therefore
	$$- C \| u \|^2_{L^2(\rho_\star)} + \frac{1}{2} \| \nabla u \|^2_{L^2(\rho_\star)} \le \| \nabla^* u \|_{L^2(\rho_\star)}^{2} \le C \| u \|^2_{L^2(\rho_\star)} + \frac{3}{2} \| \nabla u \|^2_{L^2(\rho_\star)} \, ,$$
	which makes it possible to conclude.
	\end{proof}

	\subsection{Proof of Theorem~\ref{thm:steady} and \ref{thm:steady_state}}\label{scn:proofsteady}
	
	Let $F_\star$ be a smooth and rapidly decaying solution of \eqref{eq:steadySCVFP}. Then observe that $G_\star(x,v) = e^{-|v|^2/2-V(x)-\Psi_{F_\star}(x)}$ also solves \eqref{eq:steadySCVFP}, in which $\Psi_{F_\star}$ is considered as a given potential. Then by multiplying the equation satisfied by $F_\star$ by $F_\star/G_\star$ and integrating, one obtains \[\|\nabla_v(F_\star/G_\star)\|_{L^2(G_\star)}^2 = 0\] which shows that $F_\star(x,v) = \rho_\star(x) M(v)$. By plugging back this ansatz into \eqref{eq:steadySCVFP}, it shows that $\rho_\star$ solves \eqref{eq:stationary_equation}. The rest of the proof of Theorem~\ref{thm:steady} is a combination of the results of Lemma~\ref{lem:properties_T} to Corollary~\ref{cor:uniqueness}. For the proof of Theorem~\ref{thm:steady_state}, just observe that Assumption~\ref{assu:confinement} and Assumption~\ref{assu:interaction} imply all the hypotheses of Theorem~\ref{thm:steady} with $N = 2$. In particular, since $L^p \cap \dot{W}^{1,q}$ embeds into $L^\infty$ for $q>d$, the assumption on $\KK^* \left(e^{-V}\right)$ is automatically satisfied and $\zeta = \zeta(\overline{\kappa}, R_V)$. Finally the functional inequalities of Section~\ref{scn:functional_inequalities_steady_state} hold and 
	we also point out that, all constants depending on $\overline{\kappa}$ can be considered to depend on the larger bound $\overline{\kappa}_{\text{max}}$.
	
	\section{Linearized VFP}
	\label{scn:linearized}
	
	In this section, we study decay and regularization properties of the linearized flow of the equation \eqref{eq:SCVFP} around the equilibrium $F_\star = \rho_\star M$ of Section~\ref{scn:steady_state}, where $M(v) = (2 \pi)^{-d/2} \exp\left( - |v|^2 / 2 \right)$. For that we introduce a perturbation $f$ satisfying
	\[
	f = \frac{F-F_\star}{F_\star}.
	\]
	The functional framework will rely on the Hilbert space $L^2(F_\star)$. In this space the adjoints of $\nabla_x$ and $\nabla_v$ are respectively
	$$\nabla_x^* = -\nabla_x + \nabla_x V_\star \quad \text{and} \quad \nabla_v^* = -\nabla_v + v \, .$$
	We introduce the subspace 
	\[
	\HH_0 = \left\{f\in L^2(F_\star)\text{ such that } \int_{\R^{2d}} f\,F_\star\,\d x\,\d v = 0\right\},
	\] and the orthogonal projection in $L^2(F_\star)$ onto local equilibrium
	\[
	\Pi f(x) = \int_{\R^d} f(x,v) M(v) \d v=  \frac{\rho_f(x)}{\rho_\star(x)} \quad\text{with}\quad\rho_f(x) = \int_{\R^d} f(x,v)\,F_\star(x,v)\d v.
	\]
	Observe that $\rho_f = \rho_{\Pi f}$. Let us write the linearized equation. By plugging $F = F_\star(1+f)$  into \eqref{eq:SCVFP}, one finds the linearized VFP equation
	\begin{equation}
	\label{eq:linearized_source}
	\begin{cases}
	( \partial_t + \Lambda ) f(t) + v \cdot \nabla_x \psi_f(t) = \nabla_v^* \varphi(t), \\
	\Lambda = \nu \nabla_v^* \nabla_v + v \cdot \nabla_x - \nabla_x V_\star \cdot \nabla_v\\
	\psi_f = \KK\rho_f\,, \\
	\displaystyle f(0, x, v) = f_\ini(x, v) \,,
	\end{cases}
	\end{equation}
	with $F_\ini = F_\star(1+f_\ini)$. In the original equation the right-hand side is given with $\varphi = - f \nabla_x \psi_f$.	However, in this section, we will assume that it is a given source. In the following lemmas and propositions, we assume that the general hypotheses of the paper, Assumption~\ref{assu:confinement} and Assumption~\ref{assu:interaction}, are satisfied. We will write $C_\star$ for any positive constant depending only on the steady state and therefore on $\overline{\kappa}_{\text{max}}$, $R_V$ and $\theta$ only. Note that $- \Lambda$ generates a $\CC^0$--semigroup on  $L^2(F_\star)$ (see \cite[Proposition 5.5]{hellfer_2005_hypoelliptic}), and thus $- ( \Lambda + v \cdot \nabla_x \psi_{ (\cdot) } )$ as well since $v \cdot \nabla \psi_{(\cdot)}$ is a bounded operator in $L^2(F_\star)$ by virtue of Lemma \ref{lem:bounds_psi_weight}.
	
	\subsection{Preliminary properties}
	The following bounds are a translation of the assumption on $\KK$ to the current $L^2(F_\star)$ framework and are a consequence of Hölder's inequality. We use the notation $\psi^\alpha_f = \KK^\alpha\rho_f$ with $\alpha = \text{o}, \text{e}$.
	\begin{lem}
	\label{lem:bounds_psi_weight}
	The following upper bounds hold for $\alpha= e, o$
	\begin{gather}
	\label{eq:upper_bound_psi_weight}
	\| \psi_f^\alpha \|_{L^2(F_\star)} \le C_\star \overline{\kappa}^\alpha \| \Pi f \|_{L^2(F_\star)}, \\
	\label{eq:upper_bound_grad_psi_weight}
	\| \nabla_x \psi_f^\alpha \|_{L^2(F_\star)} + \| \nabla_x \psi_f^\alpha \|_{L^q} \le  C_\star  \overline{\kappa}^\alpha \| \Pi f \|_{L^2(F_\star)},
	\end{gather}
	as well as the lower bound:
	\begin{equation}
	\label{eq:lower_bound_psi}
	\la \psi_f , \rho_f \ra_{L^2} = \la \psi_f^\text{e}, \rho_f \ra_{L^2} \ge -  (C_\star\theta + (1-\theta))\underline{\kappa}^\text{e}\| \Pi f \|^2_{L^2(F_\star)} \, .
	\end{equation}
	\end{lem}
	
	The next lemma is a consequence of the Poincaré inequalities and other results of Section~\ref{scn:functional_inequalities_steady_state}.
	
	\begin{lem}
	Let $f\in \HH_0\subset  L^{2}(F_\star)$ and $g\in L^{2}(F_\star)$. Then there is $C = C(C_\star)$ such that
	\begin{align}
	\|\Pi f\|_{L^{2}(F_\star)}&\leq C\|\nabla_x f\|_{L^{2}(F_\star)},\label{eq:poincare_pi_zeroaverage}\\
	\|(\id-\Pi) g\|_{L^{2}(F_\star)}&\leq \|\nabla_v g\|_{L^{2}(F_\star)},\label{eq:poincare_idmoinspi}\\
	\|\Pi g\|_{L^{2}(F_\star)}\leq\|g\|_{L^{2}(F_\star)}&\leq C\left(\|\nabla_x g\|_{L^{2}(F_\star)} + \|\nabla_v g\|_{L^{2}(F_\star)}\right),\label{eq:poincare_pi}
	\end{align}
	Moreover
	\begin{equation}\label{eq:charac_norm_dvstar}
	\|\nabla_v^*g\|_{L^{2}(F_\star)}^2 = d \, \|g\|_{L^{2}(F_\star)}^2 + \|\nabla_v g\|_{L^{2}(F_\star)}^2 \, , \qquad \|\nabla_v^* \cdot G\|_{L^{2}(F_\star)}^2 \le \|G\|_{L^{2}(F_\star)}^2 + \|\nabla_v G\|_{L^{2}(F_\star)}^2 \, ,
	        \end{equation}
	        where, for a vector field $G$, we denoted $\nabla_v^* \cdot G = \sum_{i=1}^d \partial_{v_i}^* G_i$; the second relation is an equality when $G = \nabla_v g$.
	Finally, for any $\eps>0$ there is $C_\eps = C_\eps(C_\star)$ such that
	\begin{equation}\label{eq:poincare_hessianVstar}
	\||\nabla^2V_\star|g\|_{L^{2}(F_\star)}\leq C_\eps\|g\|_{L^{2}(F_\star)} + \eps\|\nabla_xg\|_{L^{2}(F_\star)}
	\end{equation}
	
	\end{lem}

	\subsection{$L^2$-hypocoercivity}
	
	In the following we want to apply the $L^2$ hypocoercivity techniques of \cite{DMS}. Taking inspiration from \cite{ADLT}, we will introduce a suitable norm which will allow us to include some part of the interaction term $v\cdot\nabla\psi_f$ into the dissipative mechanisms of the equation instead of a mere perturbation of the hypocoercive operator $\Lambda$. Accordingly, let us rewrite the equation \eqref{eq:linearized_source} as
	\begin{equation}
	\label{eq:PSCVFP_order_0}
	\begin{cases}
	( \partial_t + T - L ) f  = - v \cdot \nabla_x \psi^\text{o}_f + \nabla_v^* \varphi, \\
	f(0, x, v) = f_\ini(x, v),
	\end{cases}.
	\end{equation}
	where the operators $T$ and $L$ are defined as
	$$L := -\nu \nabla_v^* \nabla_v \quad \text{and} \quad T := v \cdot \nabla_x - \nabla_x V_\star \cdot \nabla_v + v \cdot \nabla_x \psi^\text{e}_f \, .$$

	\subsubsection{The linearized free energy norm}

	We endow the subspace  $\HH_0$ with a norm for which $L$ and $T$ will be shown to be respectively symmetric and skew-symmetric. Let us consider the twisted inner product
	\begin{equation}
	\label{eq:def_twisted_L2_product}
	\lla f, g \rra := \la f, g \ra_{L^2 \left( F_\star \right) } + \la \psi^\text{e}_f , \rho_g \ra_{L^2 } = \la f + \psi^\text{e}_f, g \ra_{L^2 \left( F_\star \right) },
	\end{equation}
	and the corresponding (squared) norm
	\begin{equation}
	\label{eq:def_twisted_L2_norm}
	\Nt{f}^2 := \lla f, f \rra = \iint_{\R^{2d}}  f(x,v)^2 F_\star(x,v)\d x\d v + \int_{\R^{d}} (\KK^\text{e}\rho_f)(x)\rho_f(x) \d x.\
	\end{equation}
	\begin{rem}
	In the case of a purely symmetric kernel $k=k^\text{e}$, observe that \eqref{eq:def_twisted_L2_norm} is related to the natural Lyapunov functional of the equation: the free energy \eqref{eq:freeenerg}. More precisely, if one plugs in $F = F_\star(1+\eps f)$ and lets $\eps \to0$, \eqref{eq:def_twisted_L2_norm} is the main order contribution. That is why we refer to it as the linearized free energy norm.   This observation is a generalization to arbitrary kernels of the one made in \cite{ADLT}. 
	\end{rem}	
	
	\subsubsection{The Dolbeault-Mouhot-Schmeiser functional $\EE_0$}
	For the new inner product \eqref{eq:def_twisted_L2_norm} we use the superscript $\dagger$ to denote the adjoint of an operator for the inner product \eqref{eq:def_twisted_L2_product}. Following \cite{DMS}, we introduce, for some $\eps \in (0, 1/2)$, the following functional:
	$$\EE_0[ f ] := \frac{1}{2} \Nt{f}^2 + \eps \lla A f, f \rra \quad \text{where} \quad A := \left(  \id + (T \Pi )^\dagger T \Pi \right)^{-1} \left( T \Pi \right)^\dagger \, .$$
	On the one hand, we will show that the functional $\EE_0$ satisfies a Lyapunov inequality, and on the other hand that both quantities $\EE_0^{\frac12}$ and $\Nt{\cdot}$ are equivalent to the norm of $L^2(F_\star)$.
	\begin{lem}
	\label{lem:Hilbert_space_structure}
	For a small enough $\delta^\text{e}=\delta^\text{e}(C_\star, \theta)>0$, assuming that $\underline{\kappa}^\text{e}< \delta^\text{e}$, one has for some $C = C(\overline{\kappa}^\text{e}, C_\star)$ that
	$$\frac{1}{2} \| f \|^2_{ L^2(F_\star) } \le \Nt{f}^2 \le C \| f \|^2_{ L^2(F_\star) } \, .$$
	The inner product is compatible with the micro--macro decomposition in the sense that
	$$\Pi^\dagger = \Pi \quad \text{and} \quad \ker(L)^{ \perp_{\HH_0} } = \ker(L)^{ \perp_{L^2(F_\star)} } ,$$
	and furthermore
	\begin{equation}
	\label{eq:twisted_inner_product_coincide_microscopic}
	\lla (\id - \Pi) f, g \rra = \la (\id - \Pi) f, g \ra_{ L^2(F_\star) }
	\end{equation}
	Finally, the diffusion and transport operators $L$ and $T$ are respectively symmetric and skew-symmetric:
	$$T^\dagger = - T \quad \text{and} \quad L = L^\dagger.$$
	\end{lem}
	
	\begin{proof}
	This follows from Lemma~\ref{lem:bounds_psi_weight}, and more precisely \eqref{eq:upper_bound_psi_weight} and \eqref{eq:lower_bound_psi}. The properties of $T, L$ and $\Pi$ can be proved exactly as in \cite[Lemma 15 and 16]{ADLT}. 
	\end{proof}

	\subsubsection{Properties of $T$, $L$, $A$ and $\Pi$}
	We now prove some properties concerning the previously introduced operators. For properties that can be proved exactly as in \cite{ADLT} we sketch the proof at most. We only detail proofs when they differ substantially or when they concern additional properties not proved in \cite{ADLT}.
	
	\begin{lem}
	\label{lem:composition_rho_T}
	The operator $T$, as an operator of $\HH_0$, satisfies the parabolic dynamics property:
	$$\Pi T \Pi = 0 , $$
	and, more precisely, there holds for any $f \in \HH_0$
	$$T \Pi f = (\id - \Pi) T \Pi f = v \cdot \nabla_x ( \Pi f + \psi^\text{e}_f )$$
	as well as
	$$\Pi T f = \Pi T (\id - \Pi ) f = - \nabla_x^* \Pi (v f).$$
	Furthermore, assuming $\underline{\kappa}^\text{e}< \delta^\text{e}(C_\star,\theta)$, there is some $\lambda_M(C_\star,\theta)>0$ such that the following macroscopic coercivity estimate hold, namely
	\begin{equation}
	\label{eq:macroscopic_coercivity}
	\| T \Pi f \| \ge \lambda_M \| \Pi f \|.
	\end{equation}
	\end{lem}
	
	\begin{proof}
	Performing the same computations as \cite[Lemma 18]{ADLT} with the Poincaré inequality \eqref{eq:Poincare_rho_star}, one has
	$$\| T \Pi f \|^2_{ L^2(F_\star) } \ge \frac{ 1}{C_{\star}}  \left( \|  \Pi f \|^2_{L^2(\rho_\star)} + 2 \int_{\R^d} \psi^\text{e}_f(x) \rho_f(x)  \d x \right)$$
	thus, using the definition of $\Nt{\cdot}$ and the lower bound \eqref{eq:lower_bound_psi}, we have for some $C=C(\theta,C_\star)$
	\begin{equation*}
	\| T \Pi f \|_{L^2(F_\star)}^2 \ge \frac{1}{C_\star}  ( 1 -  \underline{\kappa}^\text{e}C) \| \Pi f \|_{L^2(\rho_\star)}^2,
	\end{equation*}
	which makes it possible to conclude.
	\end{proof}
	
	This next lemma can be proved following  \cite[Lemma 20]{ADLT}, the only novelty is the last inequality.
	
	\begin{lem}
	\label{lem:A_bounded}
	The auxiliary operator $A$ is such that
	$$A = \Pi A (\id - \Pi) : \ker(L)^\perp \rightarrow \ker(L)$$
	and satisfies the following boundedness properties:
	$$ \Nt{ A L f } \le \frac{\nu}{ 2 } \Nt{f}, \qquad \Nt{T A f} \le \Nt{f} , \qquad \Nt{A f} \le \frac{1}{2} \Nt{f}.$$
	         Finally, there holds for some $C = C(\theta, C_\star, \overline{\kappa}^{\text{e}}) > 0$
	        \begin{equation}
	            \label{eq:A_grad_star}
	            \Nt{ A (\nabla_v^* f) } \le C \Nt{f} \ .
	        \end{equation}
	\end{lem}

	    \begin{proof}
	        We only prove the last point, which is equivalent to establishing a bound for
	        $$\nabla_v A^\dagger f = \nabla_v T \Pi g \ , \qquad g := \left( \id + (T\Pi)^\dagger T \Pi \right)^{-1} f \ .$$
	        We notice that $\nabla_v T \Pi g = \nabla_x \Pi g + \nabla_x \psi^{\text{e}}_g$, and since one has
	        $$\| \nabla_x \Pi g \|_{L^2(F_\star)} = \| v \cdot \nabla_x \Pi g \|_{L^2(F_\star)} \le \| T \Pi g \|_{L^2(F_\star)} + \| \nabla_x \psi^{\text{e}}_g \|_{L^2(F_\star)}  \ ,$$
	        we deduce, using the boundedness \eqref{eq:upper_bound_grad_psi_weight} of $\nabla_x \psi^{\text{e}}$
	        \begin{align*}
	            \| \nabla_v A^\dagger f \|_{L^2(F_\star)} & \le \| A^\dagger f \|_{L^2(F_\star)} + 2 \| \nabla_x \psi^{\text{e}}_g \|_{L^2(F_\star)} \\
	            & \le \| A^\dagger f \|_{L^2(F_\star)} + 2 C_\star \overline{\kappa}^{\text{e}} \| g \|_{L^2(F_\star)} \ .
	        \end{align*}
	        Since $\Nt{g} \le \Nt{f}$ and $\Nt{A^\dagger f} \le \frac12 \Nt{f}$, we conclude by norm equivalence.
	    \end{proof}

	\begin{lem}
	\label{lem:AT_bounded}
	The composition $A T$ satisfies for some $C = C( C_\star, \overline{\kappa}^\text{e})$ the bound
	$$\Nt{A T (\id - \Pi ) f} \le C \Nt{ (\id - \Pi) f }  \, .$$
	\end{lem}

	\begin{proof}
	In this proof, we denote by $C = C( C_\star, \overline{\kappa}^\text{e})$ a constant that may change from line to line. The proof follows the lines of the proof of \cite[Lemma 21]{ADLT}, so we focus on the most intricate estimate \eqref{eq:weight_g_H1_psig_Hd1}. Let us however point out that this boundedness result is proved by establishing it for the adjoint $ (AT (\id - \Pi))^\dagger$ and reducing the problem to
	\begin{equation*}
	\| ( A T (\id - \Pi ) )^\dagger f \|_{L^2(F_\star)}^2 \le C \left( \| g \|^2_{L^2(F_\star)}  + \| f \|_{L^2(F_\star)}^2 + \| \nabla_x w_g | \nabla_x V_\star | \|^2_{L^2(F_\star)} \right) \, ,
	\end{equation*}
	where $g \in \ker(L)$ is characterized by the elliptic equation
	\begin{equation}
	\label{eq:elliptic_w_g}
	g + \nabla_x^* \nabla_x w_g = f \qquad \text{with} \qquad w_g := g + \psi^\text{e}_g \, .
	\end{equation}
	Taking the $L^2(F_\star)$--inner product against $g + \psi^\text{e}_g$, one can prove the control
	\begin{equation}
	\label{eq:g_H1_psig_Hd1}
	\Nt{g}^2 + \| \nabla_x g \|^2_{L^2(F_\star)} + \| \nabla_x w_g \|^2_{L^2(F_\star)} \le C \Nt{f}^2 \, .
	\end{equation}
	Let us explain how to establish the bound
	\begin{equation}
	\label{eq:weight_g_H1_psig_Hd1}
	\| \nabla_x w_g  | \nabla_x V_\star | \|^2_{L^2(F_\star)}  \le C \Nt{f}^2 \, .
	\end{equation}
	Taking the $L^2\left( F_\star \right)$-inner product of \eqref{eq:elliptic_w_g} with $g | \nabla_x V_\star |^2$,
	$$ \left\la \nabla_x g, \nabla_x \left( g | \nabla_x V_\star |^2 \right) \right\ra_{L^2(F_\star)} + \left\la \nabla_x \psi^\text{e}_g, \nabla_x \left( g | \nabla_x V_\star |^2 \right) \right\ra_{L^2(F_\star)} = \la g - f, g | \nabla_x V_\star |^2 \ra_{L^2(F_\star)},$$
	thus, using the Leibniz rule, we get
	\begin{align*}
	\Big\| \nabla_x g \left| \nabla_x V_\star \right| \Big\|^2_{ L^2(F_\star) }
	= & -  \left\la \nabla_x g , g \nabla_x \left( | \nabla_x V_\star |^2 \right) \right\ra_{ L^2(F_\star) } - \left\la \nabla_x \psi^\text{e}_g , g \nabla_x \left( | \nabla_x V_\star |^2 \right) \right\ra_{ L^2(F_\star) } \\
	&  - \left\la \nabla_x \psi^\text{e}_g , \nabla_x g | \nabla_x V_\star |^2 \right\ra_{ L^2(F_\star) } \\
	& + \Big\| g | \nabla_x V_\star | \Big\|^2_{ L^2(F_\star) } - \left\la f, g | \nabla_x V_\star |^2 \right\ra_{ L^2(F_\star) } =: \sum_{k = 1}^{5} X_k.
	\end{align*}
	The term $X_1$ is estimated using that $\nabla_x \left( | \nabla_x V_\star |^2 \right)  \le C | \nabla_x V_\star |^2 +C$ from  \eqref{eq:hess_grad_V_star}:
	\begin{align}
	\notag
	| X_1 |
	\le & C  \int_{ \R^d } | \nabla_x g | | g | \rho_\star \d x + C \int_{ \R^d } | \nabla_x g | | g | | \nabla_x V_\star |^2 \rho_\star \d x
	\end{align}
	and thus, for any $\delta \in (0, 1)$
	\begin{equation}
	\label{eq:est_X1}
	| X_1| \le C \| g \|^2_{ L^2(F_\star) } + C \| \nabla_x g \|^2_{ L^2(F_\star) } + \frac{C}{\delta} \| g | \nabla_x V_\star | \|^2_{ L^2(F_\star) } + \delta \Big\| \nabla_x g | \nabla_x V_\star | \Big\|^2_{ L^2(F_\star) }.
	\end{equation}
	Similarly, the term $X_2$ is controlled using Hölder's inequality for $\frac{1}{2} = \frac{1}{q} + \frac{1}{s}$ followed by \eqref{eq:upper_bound_grad_psi_weight} and \eqref{eq:moment_rho_star}:
	\begin{align}
	\notag
	| X_2 |
	\le & C \int_{ \R^d } | \nabla_x \psi^\text{e}_g | | g | \rho_\star \d x + C \int_{ \R^d } | \nabla_x \psi^\text{e}_g | \times | g | | \nabla_x V_\star | \rho_\star^{\frac{1}{2}} \times | \nabla_x V_\star | \rho_\star^{\frac{1}{2}} \d x \\
	\notag
	\le & C \| \nabla_x \psi^\text{e}_g \|_{ L^2(F_\star) } \| g \|_{L^2(F_\star)} + C \| \nabla_x \psi^\text{e}_g \|_{L^q} \| g | \nabla_x V_\star | \|_{ L^2(F_\star) } \| | \nabla_x V_\star |^2 \rho_\star \|_{L^{s / 2}}^{\frac 12}  \\
	\label{eq:est_X2}
	\le & C \left( \| g \|^2_{ L^2(F_\star) } + \| g | \nabla_x V_\star | \|^2_{ L^2(F_\star) } \right).
	\end{align}
	The term $X_3$ is controlled similarly:
	\begin{align}
	\notag
	| X_3 | \le & \int_{ \R^d } | \nabla_x \psi^\text{e}_g | | \nabla_x g | \left| \nabla_x V_\star \right|^2 \rho_\star \d x  \\
	\notag
	\le & C \| \nabla_x \psi^\text{e}_g \|_{L^q} \| \nabla_x g | \nabla_x V_\star | \|_{ L^2(F_\star) } \| | \nabla_x V_\star |^2 \rho_\star \|_{L^{s / 2}}^{1/2} \\
	\label{eq:est_X3}
	\le & \frac{C}{\delta} \| g \|_{L^2(F_\star) }^2 + \delta \left\| \nabla_x g | \nabla_x V_\star | \right\|_{ L^2(F_\star) }^2.
	\end{align}
	The terms $X_4$ and $X_5$ are simply controlled as
	\begin{equation}
	\label{eq:est_X4_5}
	| X_4 | + | X_5 | \lesssim \Big\| g | \nabla_x V_\star | \Big\|^2_{ L^2(F_\star) } + \frac{1}{\delta} \| f \|^2_{ L^2(F_\star) } + \delta \Big\| g | \nabla_x V_\star |^2 \Big\|^2_{ L^2(F_\star) } .
	\end{equation}
	Putting together \eqref{eq:est_X1}--\eqref{eq:est_X4_5}, using the weighted Poincaré inequalities \eqref{eq:Poincare_gV_rho_star}--\eqref{eq:Poincare_gV2_rho_star} (recall that $g$ has zero mean), and using the bound \eqref{eq:g_H1_psig_Hd1} we have
	$$\Big\| \nabla_xg  | \nabla_x V_\star | \Big\|^2_{L^2(F_\star)}  \le C \Nt{f}^2 \, .$$
	Finally, as for \eqref{eq:est_X2}--\eqref{eq:est_X3}, one proves a similar weighted Sobolev estimate for $\psi_g$, which together with the bound above yields \eqref{eq:weight_g_H1_psig_Hd1}. This concludes the proof.
	
	\end{proof}
	
	\subsubsection{Properties of the functional $\EE_0$ }
	
	\begin{lem}
	\label{lem:Lyapunov_order_0}
	For $\eps \in (0, 1/2]$, the functional $\EE_0$ is equivalent to the norm $\Nt{\cdot}$ in the sense that
	$$\frac{1}{4} \Nt{f}^2 \le \EE_0[f] \le \frac{3}{4} \Nt{f}^2.$$
	Furthermore, under the smallness assumptions
	$$\underline{\kappa}^\text{e} < \delta^\text{e}(C_\star, \theta) \qquad \text{and} \qquad \overline{\kappa}^\text{o}< \delta^\text{o}(C_\star, \overline{\kappa}^\text{e}, \underline{\kappa}^\text{e}, \nu )$$
	for some appropriate choice of $\eps$, there exists $\lambda_0, \mu$ and $C$ depending on $C_\star, \overline{\kappa}^\text{e}, \underline{\kappa}^\text{e}, \overline{\kappa}^\text{o}$, $\theta$ and $\nu$ such that any solution to \eqref{eq:linearized_source} satisfies the differential inequality
	$$\frac{\d}{\d t} \EE_0[f] + \lambda_0 \, \EE_0[f] + \mu \| \nabla_v f \|^2_{L^2(F_\star)} \le C \| \varphi \|_{L^2(F_\star)}^2 \, .$$
	\end{lem}
	
	\begin{proof}
	The equivalence of norms follows from the bound of $A$ provided by Lemma \ref{lem:A_bounded} which gives
	$$\frac{1 - \eps}{2} \Nt{f}^2 \le \EE_0[f] \le \frac{1 + \eps}{2} \Nt{f}^2.$$
	Differentiating $\EE_0\left[ f \right]$, and using that
	\begin{gather*}
	v \cdot \nabla_x \psi^\text{o}_f \in \ker(L)^\perp, \quad L f \in \ker(L)^\perp, \quad \text{and} \quad \nabla^*_v \varphi \in \ker(L)^\perp, \\
	T^\dagger = -T, \quad \text{and} \quad A f \in \ker(L)
	\end{gather*}
	one has the identity
	\begin{align*}
	\frac{\d}{\d t} \EE_0[f] = & \lla \partial_t f, f \rra + \eps \lla A \partial_t f, f \rra + \eps \lla A f, \partial_t f \rra \\
	= & \lla L f, f \rra -  \lla v \cdot \nabla_x \psi^\text{o}_f , f \rra \\
	& - \eps \lla A T f, f \rra + \eps \lla A L f, f \rra - \eps \lla A \left( v \cdot \nabla_x \psi^\text{o}_f  \right) , f \rra \\
	& - \eps \lla A f, T f \rra \\
	& + \lla f , \nabla_v^* \varphi \rra + \eps \lla A f , \nabla_v^* \varphi \rra + \eps \lla A \nabla_v^* \varphi , f \rra =: \sum_{k = 1}^{9} X_k.
	\end{align*}
	The term $X_1$ is estimated using the microscopic coercivity estimate \eqref{eq:poincare_idmoinspi} and the fact that $\| \cdot \|_{L^2(F_\star)} = \Nt{ \cdot }$ on $\ker(L)^\perp$:
	$$X_1 \le - \frac{\nu}{2} \Nt{ (\id - \Pi) f }^2 - \frac{\nu}{2} \| \nabla_v f \|^2_{L^2(F_\star)},$$
	Noting that $(T \Pi)^\dagger T \Pi$ is a self-adjoint operator whose spectrum lies in $[\lambda_M, +\infty)$ by Lemma \ref{lem:composition_rho_T}, one has by functional calculus
	$$\lla A T \Pi f, \Pi f \rra \ge \frac{\lambda_M}{1 + \lambda_M} \Nt{\Pi f }^2,$$
	thus, the term $X_3$ is estimated for any $\delta \in (0, 1)$ as
	\begin{align*}
	X_3 = - \eps \lla A T f, \Pi f \rra & \le - \frac{\eps \lambda_M}{1 + \lambda_M} \Nt{ \Pi f }^2 - \eps \lla A T (\id - \Pi) f, \Pi f \rra \\
	& \le -  \eps \left( \frac{ \lambda_M}{ 1 + \lambda_M  } -\delta \right) \Nt{ \Pi f }^2 + \frac{\eps C}{\delta} \Nt{ (\id - \Pi) f }^2
	\end{align*}
	for some constant $C = C(\overline{\kappa}^\text{e}, C_\star )$ by Lemma \ref{lem:AT_bounded}. Similarly, the boundedness of $T A$ from Lemma \ref{lem:A_bounded} yields
	\begin{align*}
	X_6 & = \eps \lla T A (\id - \Pi) f, f \rra \\
	& \le \eps \Nt{ T A (\id - \Pi) f } \Nt{ \Pi f } + \eps \Nt{ T A (\id - \Pi) f } \Nt{| (\id - \Pi ) f } \\
	& \le \delta \eps \Nt{ \Pi f }^2 + \eps \left( 1 + \frac{1}{4 \delta} \right) \Nt{ (\id - \Pi) f }^2,
	\end{align*}
	and that of $AL$
	\begin{align*}
	X_4  = \eps \lla AL (\id - \Pi) f, \Pi f \rra & \le \eps \delta \Nt{ \Pi f }^2 + \frac{\eps}{4 \delta} \Nt{ AL (\id - \Pi) f }^2 \\
	& \le \eps \delta \Nt{ \Pi f }^2 + \frac{\eps \nu^2}{16 \delta} \Nt{ (\id - \Pi) f }^2.
	\end{align*}
	The remaining force terms are controlled using the fact that $\lla \cdot , \cdot\rra$ coincide with the~$L^2(F_\star)$--inner product on microscopic distributions, as well as the upper bound \eqref{eq:upper_bound_grad_psi_weight} on $\psi^\text{o}_f$ :
	\begin{align*}
	|X_2| = |\la v \cdot \nabla_x \psi^\text{o}_f , (\id - \Pi) f \ra_{L^2(F_\star)}|
	\le C \overline{\kappa}^\text{o}\left( \Nt{ \Pi f }^2+ \Nt{ (\id - \Pi) f }^2 \right)
	\end{align*}
	where we used the equivalence of norms in the last line,  and similarly, using that $A$ has operator norm $1/2$ from Lemma \ref{lem:A_bounded}
	\begin{align*}
	|X_5| \le C \eps \overline{\kappa}^\text{o}\Nt{ \Pi f }^2 .
	\end{align*}
	Since $\nabla_v^* \varphi$ is microscopic in the sense that $\Pi (\nabla_v^* \varphi) = 0$, one has by virtue of Lemma \ref{lem:A_bounded} that $X_8 = 0$ as well as
	\[
	X_7 = \lla f , \nabla_v^* \varphi \rra =  \la (\id - \Pi) f , \nabla_v^* \varphi \ra_{L^2(F_\star) }  \le \frac{\nu}{4} \| \nabla_v f \|^2_{L^2(F_\star)} + \frac{1}{ \nu }  \| \varphi \|_{L^2(F_\star)}^2 \, .
	\]
	         Finally, the last term is estimated using \eqref{eq:A_grad_star}:
	        $$| X_9 | \le \eps \delta \Nt{\Pi f}^2 + \frac{\eps C}{\delta} \Nt{ \varphi }^2  \ ,$$
	        where we also used that $A$ takes values in $\ker(L)$.
	        
	Put together, these estimates yield
	\begin{align*}
	\frac{\d}{\d t} \EE_0[f] \le & - \frac{ \nu }{4} \| \nabla_v f \|_{L^2(F_\star)}^2 \\
	& - \left\{ \frac{\nu}{2} - \frac{ \eps C }{4 \delta} - \eps \left( 1 + \frac{1}{4 \delta} \right) - \frac{\eps \nu^2}{8 \delta} - C \overline{\kappa}^\text{o}\right\} \Nt{ (\id - \Pi) f }^2 \\
	& - \left\{ \eps \left( \frac{\lambda_M}{1 + \lambda_M} - 4 \delta \right) - (1+\eps) C \overline{\kappa}^\text{o} \right\} \Nt{ \Pi f }^2 \\
	& + C \| \varphi \|^2_{L^2(F_\star)} \, .
	\end{align*}
	Choosing $\eps = \delta^{2}$ and using the upper bound \eqref{eq:upper_bound_grad_psi_weight}, this simplifies for some $C = C( \overline{\kappa}^\text{e}, C_\star )$ as
	
	\begin{align*}
	\frac{\d}{\d t} \EE_0[f] \le & - \frac{ \nu }{4} \| \nabla_v f \|_{L^2(F_\star)}^2 \\
	& - \left\{ \frac{\nu}{2} - \delta C - 2 \delta - \delta \nu - C \overline{\kappa}^\text{o}\right\} \Nt{ (\id - \Pi) f }^2 \\
	& - \left\{ \delta^2 \left( \frac{\lambda_M}{1 + \lambda_M} - 4 \delta \right) - C \overline{\kappa}^\text{o} \right\} \Nt{ \Pi f }^2 \\
	& + C \| \varphi \|^2_{L^2(F_\star)} \, .
	\end{align*}
	The result then holds taking $\delta$ small enough, depending on $C$, $\nu$ and $\lambda_M$, so that the dissipation rates of $(\id-\Pi) f$ and $\Pi f$ are of order $\nu$ and $\delta^2$ respectively, and then $\overline{\kappa}^\text{o}$ small enough, depending on $\nu$ and $\delta^2$.
	\end{proof}
	
	We can now state the main result of this section, which is an immediate consequence of Lemmas \ref{lem:Lyapunov_order_0} and \ref{lem:Hilbert_space_structure}.
	
	\begin{prop}[$L^2_{x, v}$-hypocoercivity]
	\label{prop:L2_linear_hypocoercivity}
	Under the smallness assumptions
	$$\underline{\kappa}^\text{e} < \delta^\text{e}(C_\star, \theta) \qquad \text{and} \qquad \overline{\kappa}^\text{o}< \delta^\text{o}(C_\star, \overline{\kappa}^\text{e}, \underline{\kappa}^\text{e}, \nu )$$
	there exists $C$ depending on $C_\star, \overline{\kappa}^\text{e}, \overline{\kappa}^\text{o}, \underline{\kappa}^\text{e}$ and $\nu$ such that, denoting by $\lambda_0$ the rate given by Lemma~\ref{lem:Lyapunov_order_0}, for any $\lambda \in [0, \lambda_0/2]$, any solution to \eqref{eq:linearized_source} satisfies
	\[
	\sup_{t \ge 0} e^{2 \lambda t} \| f(t) \|_{L^2(F_\star)}^2 +  \int_0^\infty e^{2 \lambda t} \| \nabla_v f(t) \|_{L^2(F_\star)}^2 \d t  \le C \left( \| f_\ini \|_{L^2(F_\star)}^2 + \int_0^\infty e^{2 \lambda t} \|\varphi(t) \|_{L^2(F_\star)}^2 \d t \right) \, .
	\]
	\end{prop}
	
	From Proposition \ref{prop:L2_linear_hypocoercivity}, one can derive a well-posedness and asymptotic stability result for the nonlinear VFP equation \eqref{eq:SCVFP} if Assumption~\ref{assu:interaction} is strengthened by assuming that $\nabla\KK:L^1\cap L^2\to L^\infty$ is a bounded operator, namely $q=\infty$. For kernels providing less integrability/regularity, one need to resort to higher order estimates and regularization properties of the equation, which is the purpose of the rest of Section~\ref{scn:linearized}.
	
	\subsection{Derivatives estimates}	
	In this section, we derive preliminary estimates on the derivatives of $f$.
	
	\begin{lem}
	Let $f$ be a solution to \eqref{eq:linearized_source}, its velocity gradient satisfies the equation
	\begin{equation*}
	\partial_t (\nabla_v f) + \Lambda ( \nabla_v f ) + \nu \nabla_v f = - \nabla_x f - \nabla_x \psi_f + \nabla_v \nabla_v^* \varphi  \, .
	\end{equation*}
	as well as the differential inequality
	\begin{equation}
	\label{eq:dissipation_grad_v}
	            \begin{split}
		\frac{1}{2} & \frac{\d}{\d t} \| \nabla_v f \|^2_{ L^2(F_\star) }  + \frac{\nu}{2} \| \nabla_v^2 f \|^2_{ L^2(F_\star) } \\
		& \le
		\delta \| \nabla_x f \|^2_{L^2(F_\star) } + C \left( \delta^{-1} \| \nabla_v f \|_{ L^2(F_\star) }^2
		+ \|\varphi \|_{L^2(F_\star)}^2 + \| \nabla_v \varphi \|_{L^2(F_\star)}^2  \right)
	\end{split}
	\end{equation}
	where $\delta \in (0, 1)$, and for some $C = C(\overline{\kappa}^\text{e}, \overline{\kappa}^\text{o}, C_\star, \nu)$.
	\end{lem}
	
	\begin{proof}
	The equation satisfied by $\nabla_v f$ is obtained by noticing that the following commutator identity holds:
	\begin{gather*}
	\left[ \nabla_v, \Lambda \right] = \nabla_x + \nu \nabla_v \,  .
	\end{gather*}
	Integrating against $\nabla_v f$, there holds
	\begin{align*}
	\frac{1}{2} \frac{\d}{\d t} \| & \nabla_v f \|^2_{ L^2(F_\star) } + \nu \| \nabla_v^2 f \|^2_{ L^2(F_\star) } + \nu \| \nabla_v f \|^2_{L^2(F_\star)} \\
	= & - \la \nabla_x f , \nabla_v f \ra_{ L^2(F_\star) } - \la \nabla_x \psi_f , \nabla_v f \ra_{ L^2(F_\star) } 
	+ \la \nabla_v^* \varphi , \nabla_v^* \nabla_v f \ra_{L^2(F_\star)}  \\
	\lesssim & \delta \left( \| \nabla_x f \|^2_{L^2(F_\star)} + \| \nabla_x \psi_f \|^2_{L^2(F_\star)} \right) + \eps \| \nabla_v^* \nabla_v f \|^2_{L^2(F_\star)}  \\
	& + \frac{1}{\delta} \| \nabla_v f \|^2_{L^2(F_\star)} + \frac{1}{\eps} \| \nabla_v^* \varphi \|^2_{L^2(F_\star)} \, .
	\end{align*}
	From there on the one hand one uses \eqref{eq:charac_norm_dvstar} with $G = \nabla_v f$ and $G = \varphi$. On the other hand, one combines \eqref{eq:upper_bound_grad_psi_weight} with the Poincaré inequality \eqref{eq:poincare_pi_zeroaverage} to get for some $C = C(C_\star, \overline{\kappa})$
	{$$\| \nabla_x \psi_f \|_{L^2(F_\star)} \le C \| \nabla_x f \|_{L^2(F_\star)} \,.$$}
	In turn, one concludes to \eqref{eq:dissipation_grad_v} by taking $\eps$ small enough with respect to $\nu$.
	\end{proof}
	
	\begin{lem}
	        \label{lem:est_gradxf}
	Let $f$ be a solution to \eqref{eq:linearized_source}, its spatial gradient satisfies the equation
	\begin{align*}
	\partial_t (\nabla_x f) + \Lambda \left(\nabla_x f  \right) = \left(\nabla_x^2 V_\star\right) \nabla_v f +  \left( v \cdot \nabla_x \psi \left[ \partial_{x_i}^* f \right] \right)_{i = 1}^{d} + \nabla_v^* \nabla_x \varphi,
	\end{align*}
	as well as the differential inequality
	\begin{equation}
	\label{eq:dissipation_grad_x}
		\frac{1}{2} \frac{\d}{\d t} \| \nabla_x f \|^2_{L^2(F_\star)} + \frac{\nu}{2} \| \nabla_v \nabla_x f \|^2_{L^2(F_\star)}  \le  C \left( \| \nabla_x f \|^2_{L^2(F_\star)} + \| \nabla_v f \|^2_{L^2(F_\star)} + \| \nabla_x \varphi \|^2_{L^2(F_\star)} \right) ,
	\end{equation}
	for some $C = C(\overline{\kappa}^\text{e}, \overline{\kappa}^\text{o}, C_\star, \nu)$.
	\end{lem}
	
	\begin{proof}
	The differential equation satisfied by $\nabla_x f$ is immediately obtained by noticing
	\begin{gather*}
	\left[ \nabla_x, \Lambda \right] = - \left( \nabla_x^2 V_\star \right) \nabla_v
	\quad \text{and} \quad
	\nabla_x \psi[f] = - \psi [ \nabla_x^* f ] \, .
	\end{gather*}
	integrating against $\nabla_x f$, we get
	\begin{align*}
	\frac{1}{2} \frac{\d}{\d t} \| & \nabla_x f \|^2_{L^2(F_\star)} + \nu \| \nabla_v \nabla_x f \|^2_{L^2(F_\star)} \\
	= &  \left\la \left(\nabla_x^2 V_\star\right) \nabla_v f , \nabla_x f \right\ra_{L^2(F_\star)} + \sum_{i = 1}^{d} \left\la v \cdot \nabla_x \psi\left[ \partial_{x_i}^* f \right] , \partial_{x_i} f \right\ra_{L^2(F_\star)} \\
	& + \la \nabla_x \varphi , \nabla_x \nabla_v f \ra_{ L^2(F_\star) }.
	\end{align*}
	Concerning the first source term, thanks to the weighted Poincaré inequality \eqref{eq:poincare_hessianVstar} with (say) $\eps =1$, one has for some $C = C(C_\star)$ which may change from line to line
	\begin{align*}
	\big\la & \left(\nabla_x^2 V_\star\right) \nabla_v f , \nabla_x f \big\ra_{L^2(F_\star)} \\
	& \leq C \left( \| \nabla_v f \|_{L^2(F_\star)} + \| \nabla_x \nabla_v f \|_{L^2(F_\star)} \right) \| \nabla_x f \|_{L^2(F_\star)} \\
	& \leq C \left( \eps \| \nabla_x \nabla_v f \|_{L^2(F_\star)}^2 + \frac{1}{\eps} \| \nabla_x f \|_{L^2(F_\star)}^2 + \| \nabla_v f \|^2_{L^2(F_\star)} \right) \, .
	\end{align*}
	Concerning the second source term, we have from  the bound \eqref{eq:upper_bound_grad_psi_weight} and the Poincaré inequality \eqref{eq:poincare_pi_zeroaverage} that for some $C = C(C_\star, \overline{\kappa}^\text{e}, \overline{\kappa}^\text{o})$
	{\begin{align*}
		\sum_{i = 1}^{d} \left\la v \cdot \nabla_x \psi\left[ \partial_{x_i}^* f \right] , \partial_{x_i} f \right\ra_{L^2(F_\star)} \le C  \| \nabla_x f \|^2_{L^2(F_\star)} 
	\end{align*}}
	Concerning the third term, we have that 
	\begin{align*}
	\la \nabla_x \varphi , \nabla_x \nabla_v f \ra_{ L^2(F_\star) } \le & \| \nabla_x \varphi \|_{L^2(F_\star)} \| \nabla_x \nabla_v f \|_{ L^2(F_\star) } \\
	\le &\frac{1}{2 \nu} \| \nabla_x \varphi \|^2_{L^2(F_\star)} +  \frac{\nu}{2} \| \nabla_x \nabla_v f \|^2_{L^2(F_\star)} \, .
	\end{align*}
	This concludes the proof.
	\end{proof}
	
	\begin{lem}
	Let $f$ be a solution to \eqref{eq:linearized_source}, the inner product of its partial gradients satisfies the differential inequality
	\begin{equation}
	\label{eq:dissipation_grad_xv}
	\begin{split}
		\frac{\d}{\d t} \la \nabla_x f & , \nabla_v f \ra_{L^2(F_\star)} + \frac{1}{2} \| \nabla_x f \|^2_{L^2(F_\star)} \le \delta  \| \nabla_x \nabla_v f \|^2_{L^2(F_\star)} + C \| f \|^2_{L^2(F_\star)} \\
		& + \frac{C}{\delta} \left( \| \nabla_v f \|_{L^2(F_\star)}^2 + \| \nabla_v^2 f \|_{L^2(F_\star)}^2 + \| \varphi \|_{L^2(F_\star)}^2 + \| \nabla_v \varphi \|_{L^2(F_\star)}^2  \right)  \, ,
	\end{split}
	\end{equation}
	for any $\delta \in (0, \delta_1)$, where the constants $C$ and $\delta_1$ depend on $\overline{\kappa}^\text{e}, \overline{\kappa}^\text{o}, C_\star$ and $\nu$.
	\end{lem}
	
	\begin{proof}
	On the one hand, there holds
	\begin{align*}
	\frac{\d}{\d t} \la \nabla_v f & , \nabla_x f \ra_{L^2(F_\star)} \\
	= & \left\la \partial_t (\nabla_v f) , \nabla_x f \right\ra_{L^2(F_\star)} + \left\la \partial_t (\nabla_x f) , \nabla_v f \right\ra_{L^2(F_\star)} \\
	= & - \left\la \Lambda (\nabla_v f) , \nabla_x f \right\ra_{L^2(F_\star)} - \| \nabla_x f \|^2_{L^2(F_\star)} - \nu \left\la \nabla_v f , \nabla_x f \right\ra_{L^2(F_\star)} \\
	& - \left\la \Lambda (\nabla_x f) , \nabla_v f \right\ra_{L^2(F_\star)} + \left\la \left( \nabla_x^2 V_\star \right) \nabla_v f , \nabla_v f \right\ra_{L^2(F_\star)}   \\
	& - \la \nabla_x \psi_f , \nabla_x f \ra_{L^2(F_\star)} + \sum_{i = 1}^{d} \left\la v \cdot \nabla_x \psi\left[ \partial_{x_i}^* f \right] , \partial_{v_i} f \right\ra_{L^2(F_\star)}  \\
	& + \la \nabla_v \nabla_v^* \varphi , \nabla_x f \ra_{L^2(F_\star) } + \la \nabla_x \nabla_v^* \varphi , \nabla_v f \ra_{L^2(F_\star) } \\
	= & - \| \nabla_x f \|^2_{L^2(F_\star)} + X_1 + X_2 + X_3 + X_4,
	\end{align*}
	where we denoted (and using that the symmetric part of $\Lambda$ is $- L = \nu \nabla^*_v \nabla_v$)
	\begin{align*}
	X_1 & {= - \la \Lambda (\nabla_v f) , \nabla_x f \ra_{L^2(F_\star)} - \la \Lambda (\nabla_x f) , \nabla_v f \ra_{L^2(F_\star)} + \left\la \left( \nabla_x^2 V_\star \right) \nabla_v f , \nabla_v f \right\ra_{L^2(F_\star)}} \\
	& = - \nu \la \nabla_v \nabla_x f , \nabla_v^2 f \ra_{L^2(F_\star)} + \left\la \left( \nabla_x^2 V_\star \right) \nabla_v f , \nabla_v f \right\ra_{L^2(F_\star)}
	\end{align*}
	as well as
	\begin{gather*}
	X_2 := - \nu \left\la \nabla_v f , \nabla_x f \right\ra_{L^2(F_\star)}  , \\
	X_3 := - \la \nabla_x \psi_f , \nabla_x f \ra + \sum_{i = 1}^{d} \left\la v \cdot \nabla_x \psi\left[ \partial_{x_i}^* f \right] , \partial_{v_i} f \right\ra_{L^2(F_\star)} , \\
	X_4 := \la \nabla_v^* \varphi , \nabla_v^* \nabla_x f \ra_{L^2(F_\star) } + \la \nabla^*_v \varphi , \nabla_x^* \nabla_v f \ra_{L^2(F_\star) } \, .
	\end{gather*}
	As in the proof of \eqref{eq:dissipation_grad_x}, one has for the term $X_1$ for some $C=C(C_\star, \nu)$ that may change from line to line
	\begin{align*}
	X_1 \lesssim & C \| \nabla_v \nabla_x f \|_{ L^2(F_\star) } \left( \| \nabla_v f \|_{ L^2(F_\star) } + \nu \| \nabla_v^2 f \|_{ L^2(F_\star) } \right)  + C \| \nabla_v f \|^2_{ L^2(F_\star) } \\
	\lesssim & \delta \| \nabla_v \nabla_x f \|^2_{L^2(F_\star)} + \frac{1}{\delta} \left( \| \nabla_v f \|^2 + \| \nabla^2_v f \|^2_{L^2(F_\star)} \right)
	\end{align*}
	For the term $X_2$, one has
	\begin{align*}
	X_2 \le & \nu \| \nabla_v f \|_{ L^2(F_\star) } \| \nabla_x f \|_{ L^2(F_\star) } \le \nu^2 \| \nabla_v f \|^2_{L^2(F_\star)} + \frac{1}{4} \| \nabla_x f \|^2_{L^2(F_\star)} \, .
	\end{align*}
	As in the proof of \eqref{eq:dissipation_grad_x}, one has for the term $X_3$ for some $C=C(C_\star, \overline{\kappa}^\text{e}, \overline{\kappa}^\text{o})$ that may change from line to line
	{\begin{align*}
		X_3 \lesssim C \left( \| f \|_{L^2(F_\star)} \| \nabla_x f \|_{ L^2(F_\star) } 
		+ \| \nabla_v f \|_{L^2(F_\star)} \| \nabla_x f \|_{L^2(F_\star)} \right) \, ,
	\end{align*}}
	therefore, there holds
	\begin{align*}
	X_3 \lesssim \eps \| \nabla_x f \|^2_{L^2(F_\star)} + \frac{C}{\eps} \left(  \| f \|_{L^2(F_\star)}^2 + \| \nabla_v f \|_{L^2(F_\star)}^2  \right) \, .
	\end{align*}
	Finally, one can show for $X_4$
	{\begin{align*}
		X_4 \lesssim & \left( \| \varphi \|_{L^2(F_\star)} + \| \nabla_v \varphi \|_{L^2(F_\star)} \right) \left( \| \nabla _x f \|_{L^2(F_\star)} + \| \nabla _v f \|_{L^2(F_\star)} +\| \nabla_v \nabla_x f \|_{L^2(F_\star)} \right)  \\
		\lesssim & \delta \left( \| \nabla_x f \|^2_{L^2(F_\star)}+ \| \nabla _v f \|^2_{L^2(F_\star)} + \| \nabla_v \nabla_x f \|^2_{L^2(F_\star)} \right) + \delta^{-1} \left( \| \varphi \|^2_{L^2(F_\star)} + \| \nabla_v \varphi \|^2_{L^2(F_\star)} \right) \, .
	\end{align*}}
	Gathering all estimates and taking $\delta$ and $\eps$ small enough, we conclude the proof.
	\end{proof}

	\subsection{Mixed $H^1_{x,v}$-$L^2_{x,v}$--hypocoercivity}
	In this section, we introduce a Lyapunov functional for the linearized equation. It involves a combination of the Dolbeault-Mouhot-Schmeiser $L^2$ hypocoercivity functional and the Hérau-Nier / Villani $H^1$ hypocoercivity functional.

	Define for some small parameters $a, b, c \in (0, 1)$ and the following functional:
	$$\EE_{1,1}[ f ] := \EE_0[f ] + a \| \nabla_v f(t) \|^2_{ L^2(F_\star) } + b \la \nabla_v f , \nabla_x f  \ra_{ L^2(F_\star) } + c \| \nabla_x f \|^2_{ L^2(F_\star) } \, .$$
	
	\begin{lem}
	\label{lem:Lyapunov_order_1}
	Under the assumption $b^2 \le ac$, the functional $\EE_{1,1}$ is equivalent to the $H^1(F_\star)$--norm in the sense that, for some $M = M(C_\star, \overline{\kappa}^\text{e})$ and $m=m(a,c)$, one has 
	\begin{equation}
	\label{eq:equiv_H1_Lyapunov_1}
	m \| f \|^2_{ H^1(F_\star) } \le \EE_{1,1}[f] \le M \| f \|^2_{ H^1(F_\star) }.
	\end{equation}
	Furthermore, under the smallness assumptions
	$$\underline{\kappa}^\text{e} < \delta^\text{e}(C_\star,\theta) \qquad \text{and} \qquad \overline{\kappa}^\text{o}< \delta^\text{o} (C_\star, \overline{\kappa}^\text{e}, \underline{\kappa}^\text{e}, \nu ) \, ,$$
	for some appropriate choice of $a, b$ and $c$, there exists $\lambda_0, \mu$ and $C$ depending on $C_\star, \overline{\kappa}^\text{e}, \overline{\kappa}^\text{o}, \underline{\kappa}^\text{e}$ and $\nu$ such that any solution to \eqref{eq:linearized_source} satisfies the differential inequality
	$$\frac{\d}{\d t} \EE_{1,1}[f] + \lambda_0 \, \EE_{1,1}[f ] + \mu  \| \nabla_v f \|^2_{H^1 (F_\star) } \le C \| \varphi \|^2_{H^1(F_\star)} \, .$$
	\end{lem}
	
	\newcommand{\bmu}{\boldsymbol{\mu}}
	\begin{proof}
	First of all, note that
	\begin{equation*}
	\begin{split}
		\EE_0[f] + \frac{a}{2} & \| \nabla_v f \|^2_{ L^2(F_\star) } +\frac{c}{2} \| \nabla_x f \|^2_{ L^2(F_\star) } \\
		& \le \EE_{1,1}[f] \le \EE_0[f] + (a+b) \| \nabla_v f \|^2_{ L^2(F_\star) } + (b + c) \| \nabla_x f \|^2_{ L^2(F_\star) },
	\end{split}
	\end{equation*}
	which implies the equivalence of $\EE_{1,1}$ with the $H^1(F_\star)$--norm as stated by \eqref{eq:equiv_H1_Lyapunov_1}.  Gathering Lemma \ref{lem:Lyapunov_order_0}, \eqref{eq:dissipation_grad_v} with some $\delta_v$, \eqref{eq:dissipation_grad_xv}  with some $\delta_{xv}$ and \eqref{eq:dissipation_grad_x}, one gets the energy estimate for some $\mu = \mu(C_\star, \overline{\kappa}^\text{e}, \overline{\kappa}^\text{o}, \underline{\kappa}^\text{e})$ and $K = K(C_\star, \overline{\kappa}^\text{e}, \overline{\kappa}^\text{o}, \underline{\kappa}^\text{e})$
	\begin{align*}
	\frac{\d}{\d t} \EE_{1,1}[f] + K\EE_0[f] \left( \mu - b \right)
	& + K\| \nabla_v f \|^2_{L^2(F_\star)} \left( \mu - \frac{a}{\delta_v} - \frac{b}{\delta_{xv}} - c \right)  \\
	& + K\| \nabla_x f \|^2_{L^2(F_\star)} \left( \mu b - a \delta_v  -c \right) \\
	& + K\| \nabla_v^2 f \|^2_{L^2(F_\star)} \left( \mu a - \frac{b}{ \delta_{xv} } \right) \\
	& + K\| \nabla_x \nabla_v f \|^2_{L^2(F_\star)} (\mu c - b \delta_{xv} ) \le C \| \varphi \|_{H^1(F_\star)}^2 \, .
	\end{align*}
	Considering $c \ll b \ll 1$, it is enough to choose parameters such that
	$$\frac{b}{a} \ll \delta_{xv} \ll \frac{c}{b} \quad \text{and} \quad a \ll \delta_v \ll \frac{b}{a}$$
	and one checks that the scaling
	$$c = \eps^{21}, \quad b = \eps^{20}, \quad a = \eps^{16}, \quad \delta_v = \eps^{8}, \quad \text{and} \quad \delta_{xv} = \eps^{2}$$ 
	fits the requirement. Taking $\eps$ small enough, we conclude this proof by norm equivalence \eqref{eq:equiv_H1_Lyapunov_1}.
	\end{proof}
	
	The following hypocoercivity result is a direct consequence of Lemma \ref{lem:Lyapunov_order_1}.
	
	\begin{prop}
	\label{prop:H1_linear_hypocoercivity}
	Under the smallness assumptions
	$$\underline{\kappa}^\text{e} < \delta^\text{e}(C_\star,\theta) \qquad \text{and} \qquad \overline{\kappa}^\text{o}< \delta^\text{o}(C_\star, \overline{\kappa}^\text{e}, \underline{\kappa}^\text{e}, \nu )$$
	there exists $C$ depending on $C_\star, \overline{\kappa}^\text{e}, \overline{\kappa}^\text{o}, \underline{\kappa}^\text{e}$ and $\nu$ such that, denoting by $\lambda_0$ the rate given by Lemma~\ref{lem:Lyapunov_order_1}, for any $\lambda \in [0, \lambda_0/2]$, any solution to \eqref{eq:linearized_source} satisfies
	\[
	\sup_{t \ge 0} e^{2 \lambda t} \| f(t) \|_{H^1(F_\star)}^2 +  \int_0^\infty e^{2 \lambda t} \| \nabla_v f(t) \|_{H^1(F_\star)}^2 \d t  \le C \left( \| f_\ini \|_{H^1(F_\star)}^2 + \int_0^\infty e^{2 \lambda t} \|\varphi(t) \|_{H^1(F_\star)}^2 \d t \right) \, .
	\]
	\end{prop}
	
	\subsection{Hypoellipticity}
	In this section we modify the functional of the previous section with time weights to capture regularization properties of the VFP equation.

	Define for some small parameters $a, b, c > 0$ the time-weighted homogeneous functional
	\begin{equation*}
	\GG^{\hom}[t, f] = a t \| \nabla_v f \|^2_{L^2(F_\star)}  + b t^2 \la \nabla_v f , \nabla_x f  \ra_{L^2(F_\star)}
	+ c t^3 \| \nabla_x f  \|^2_{L^2(F_\star)} ,
	\end{equation*}
	as well as the inhomogeneous one:
	$$\GG[t, f] = \EE_0[f] + \GG^{\hom}[t, f].$$
	
	\begin{lem}
	\label{lem:Lyapunov_regularization_order_1}
	Assuming $b^2 \le ac$, there holds
	\begin{equation}
	\label{eq:equiv_reg_Lyapunov_1}
	\frac{1}{2} \GG^{\hom}[t, f] \le a t \| \nabla_v f \|^2_{L^2(F_\star)} + c t^3 \| \nabla_x f \|^2_{L^2(F_\star)} \le 2 \GG^{\hom}[t, f].
	\end{equation}
	Furthermore, under the smallness assumptions
	$$\underline{\kappa}^\text{e} < \delta^\text{e}(C_\star,\theta) \qquad \text{and} \qquad \overline{\kappa}^\text{o}< \delta^\text{o}(C_\star, \overline{\kappa}^\text{e}, \underline{\kappa}^\text{e}, \nu ) \, ,$$
	for some appropriate choice of $a, b$ and $c$, there exists $\mu$ and $C$ depending on $C_\star, \overline{\kappa}^\text{e}, \overline{\kappa}^\text{o}, \underline{\kappa}^\text{e}$ and $\nu$ such that any solution to \eqref{eq:linearized_source} satisfies the differential inequality, for $t\in(0,1]$, 
	\begin{equation}
	\label{eq:Lyapunov_hypoellipticity}
	\begin{split}
		\frac{\d}{\d t} \GG[t, f] 
		+ \mu \bigg( & \EE_0[f]
		+ \| \nabla_v f \|_{L^2(F_\star)}^2 
		+ t \| \nabla_v^2 f \|_{L^2(F_\star)}^2 \\
		& + t^2 \| \nabla_x f \|_{L^2(F_\star)}^2
		+  t^3 \| \nabla_x \nabla_v f \|_{L^2(F_\star)}^2 \bigg) \\
		& \le C \left( \| \varphi \|_{L^2(F_\star)}^2 + t \| \nabla_v \varphi \|^2_{L^2(F_\star)} + t^3 \| \nabla_x \varphi(t) \|^2_{L^2(F_\star)} \right) \, .
	\end{split}
	\end{equation}
	\end{lem}
	
	\begin{proof}
	The comparison \eqref{eq:equiv_reg_Lyapunov_1} follows from the observation that
	$$\left| b t^2 \la \nabla_v f , \nabla_x f \ra_{ L^2(F_\star) } \right| \le \frac{a t}{2} \| \nabla_v f \|^2_{L^2(F_\star)} + \frac{c t^3}{2}\| \nabla_x f \|^2_{L^2(F_\star)} \, .$$
	Combined with Lemma \ref{lem:Lyapunov_order_0}, \eqref{eq:dissipation_grad_v} with  $\delta = \delta_v t$, \eqref{eq:dissipation_grad_xv}  with  $\delta = \delta_{xv} t$ and \eqref{eq:dissipation_grad_x}, one gets the energy estimate for some $\mu = \mu(C_\star, \overline{\kappa}^\text{e}, \overline{\kappa}^\text{o}, \underline{\kappa}^\text{e})$ and  $K = K(C_\star, \overline{\kappa}^\text{e}, \overline{\kappa}^\text{o}, \underline{\kappa}^\text{e})$,
	\begin{align*}
	\frac{\d}{\d t} \GG[f] + K\EE_0[f] \left( \mu - b t^2 \right)
	& + K\| \nabla_v f \|^2_{L^2(F_\star)} \left( \mu - \frac{a}{\delta_v} - \frac{b t}{\delta_{xv}} - c t^3 - 2a \right)  \\
	& + K\| \nabla_x f \|^2_{L^2(F_\star)} \left( \mu b t^2 - a t^2 \delta_v  - c t^3 - c t^2 -3 c t^2 \right) \\
	& + K\| \nabla_v^2 f \|^2_{L^2(F_\star)} \left( \mu a t - \frac{b t}{ \delta_{xv} } \right) \\
	& + K\| \nabla_x \nabla_v f \|^2_{L^2(F_\star)} (\mu c t^3 - b t^3 \delta_{xv} ) \\
	& \le C_{a, b, c, \delta_v, \delta_{xv}} \left( \| \varphi \|_{L^2(F_\star)}^2 + t \| \nabla_v \varphi \|_{L^2(F_\star)}^2  + t^3 \| \nabla_x \varphi \|_{L^2(F_\star)}^2 \right) \, .
	\end{align*}
	The same choice of parameters as in the proof of Lemma \ref{lem:Lyapunov_order_1} lets us conclude.
	\end{proof}
	
	\begin{prop}[$H^1$--hypoellipticity]
	\label{prop:H1_linear_hypoellipticity}
	Under the smallness assumptions
	$$\underline{\kappa}^\text{e} < \delta^\text{e}(C_\star,\theta) \qquad \text{and} \qquad \overline{\kappa}^\text{o}< \delta^\text{o}(C_\star, \overline{\kappa}^\text{e}, \underline{\kappa}^\text{e}, \nu ) \, ,$$
	there exists $C$ depending on $C_\star, \overline{\kappa}^\text{e}, \overline{\kappa}^\text{o}, \underline{\kappa}^\text{e}$ and $\nu$ such that any solution to \eqref{eq:linearized_source} satisfies the short time regularization estimates:
	\begin{align*}
	\sup_{0 < t \le 1}  & \left\{ t^3 \| \nabla_x f(t) \|^2_{L^2(F_\star)} + t \| \nabla_v f(t) \|^2_{L^2(F_\star)} \right\} \\
	& + \int_0^1 \left\{ t^{2 } \| \nabla_x f(t) \|^2_{L^2(F_\star)} + \| \nabla_v f(t) \|_{L^2(F_\star)}^2 \right\} \d t \\
	& + \int_0^1 \left\{ t^3 \|  \nabla_x \nabla_v f(t) \|^2_{L^2(F_\star)} + t \| \nabla_v^2 f(t) \|^2_{L^2(F_\star)} \right\} \d t \\
	& \le C\left( \| f_\ini \|^2_{L^2(F_\star)} + \int_0^1 \left( \| \varphi(t) \|_{L^2(F_\star)}^2 + t \| \nabla_v \varphi(t) \|^2_{L^2(F_\star)} + t^3 \| \nabla_x \varphi(t) \|^2_{L^2(F_\star)} \right) \d t \right) \, .
	\end{align*}
	\end{prop}
	
	\begin{rem}
	Note that the pointwise estimate imply, in the case of spatial derivatives (and $\varphi = 0$), that
	$$\left\| t \, \nabla_x f \right\|_{L^r_t L^2_{x,v} (F_\star)} \le C_r \| f_\ini \|_{L^2_{x, v}}, \quad 1 \le r < 2 \, ,$$
	however, the integral estimate from Proposition \ref{prop:H1_linear_hypoellipticity} indicates that this holds for the endpoint $r = 2$ as well.
	\end{rem}
	
	    The following is a combination of Proposition \ref{prop:H1_linear_hypoellipticity} for $t \in (0, 1]$ and Proposition \ref{prop:H1_linear_hypocoercivity}, applied from the initial time $t=1$, for $t \ge 1$.
	
	\begin{prop}\label{prop:H1_linear_hypoellipticity_hypocoercivity}
	Under the smallness assumptions
	$$\underline{\kappa}^\text{e} < \delta^\text{e}(C_\star,\theta) \qquad \text{and} \qquad \overline{\kappa}^\text{o}< \delta^\text{o}(C_\star, \overline{\kappa}^\text{e}, \underline{\kappa}^\text{e}, \nu ) \, ,$$
	there exists $C$ depending on $C_\star, \overline{\kappa}^\text{e}, \overline{\kappa}^\text{o}, \underline{\kappa}^\text{e}$ and $\nu$ and $\lambda$ such that any solution to \eqref{eq:linearized_source} satisfies
	\begin{align*}
	\sup_{ t > 0}  \Big\{ w_3(t)^2 \| & \nabla_x f(t) \|^2_{L^2(F_\star)} + w_1(t)^2 \| \nabla_v f(t) \|^2_{L^2(F_\star)} + w_0(t)^2 \| f(t) \|_{L^2(F_\star)}^2 \Big\} \\
	& + \int_0^\infty \left\{ w_2(t)^2 \| \nabla_x f(t) \|^2_{L^2(F_\star)} + w_0(t)^2 \| \nabla_v f(t) \|_{L^2(F_\star)}^2 \right\} \d t \\
	\le C \bigg(  \| f_\ini & \|^2_{L^2(F_\star)} + \int_0^\infty \Big( w_0(t)^2 \| \varphi(t) \|_{L^2(F_\star)}^2
	+ w_1(t)^2 \| \nabla_v \varphi(t) \|^2_{L^2(F_\star)} \\
	& + w_3(t)^2 \| \nabla_x \varphi(t) \|^2_{L^2(F_\star)} \Big) \d t \bigg) \, ,
	\end{align*}
	        where we denoted $w_\sigma(t) := e^{\lambda t} \min\{ 1, t\}^{ \frac{\sigma}{2} }$ for a given $\lambda \in [0, \lambda_0/2)$, $\lambda_0$ being the rate given by Lemma~\ref{lem:Lyapunov_order_1}, and where $C$ depends also on $(\lambda_0 - 2\lambda)^{-1}$.
	\end{prop}
	
	From here one could interpolate the results of Proposition~\ref{prop:H1_linear_hypoellipticity_hypocoercivity}	with those of Propositions~\ref{prop:H1_linear_hypocoercivity} to quantify the regularization and decay of a solution to the nonlinear equation \eqref{eq:SCVFP} with an initial data in $H^{s}_{x,v}(F_\star)$, for large enough $s\in[0,1]$. This is even possible under a slightly weaker assumption than Assumption~\ref{assu:confinement} on the confining potential, namely \eqref{eq:smoothness_V} holding only for some given $\eps$. At the price of \eqref{eq:smoothness_V} holding for $\eps$ arbitrarily small, we will be able to lower the regularity assumptions on the initial data to $H^{s}_{x}L^2_v(F_\star)$.

	\subsection{$H^1_{x} L^2_{v}$ hypocoercivity}
	\label{scn:H1xL2v_hypocoercivity}
	
	In this section, we apply the Dolbeault-Mouhot-Schmeiser $L^2$--hypocoercivity techniques, in the linearized free energy norm, to a first order $x$ derivative of the solution. More precisely we shall consider $\nabla_x^*f$ instead of $\nabla_xf$ as only the former belongs to $\HH_0$.
	
	\begin{lem}\label{lem:estimate_twisted_gradx}
	Let $f$ be a solution to \eqref{eq:linearized_source}, its twisted spatial gradient $\nabla_x^*f$ satisfies the equation
	\begin{align*}
	\partial_t (\nabla_x^* f) + (T - L) \left(\nabla_x^* f \right) = -(v\cdot\nabla_x)\psi_{\nabla_x^*f}^{\text{o}}+\nabla_v^* \nabla_x^* \varphi +X_1 + X_2 \, ,
	\end{align*}
	where we denoted
	\begin{equation*}
	X_1 := \left(\nabla_x^2 V_\star\right) \nabla_v^* f, \qquad X_2 := - \nabla_xV_\star \left( v \cdot \nabla_x \psi_f \right).
	\end{equation*}
	Furthermore, under the smallness assumptions
	$$\underline{\kappa}^\text{e} < \delta^{\text{e}}(C_\star,\theta)\quad \text{and} \quad \overline{\kappa}^\text{o}< \delta^\text{o}(C_\star, \overline{\kappa}^\text{e}, \underline{\kappa}^\text{e}, \nu),$$
	there exists $\lambda_0, \mu$ and $C$ depending on $C_\star, \underline{\kappa}^\text{e}, \overline{\kappa}^\text{e}, \overline{\kappa}^\text{o}$ and $\nu$ such that $f$ satisfies the differential inequality
	\begin{equation*}
	\begin{split}
		\frac{\d}{\d t} \EE_0\left[ \nabla_x^* f \right] & + \lambda_0 \, \EE_0\left[ \nabla_x^* f \right]+ \mu \| \nabla_v \nabla_x^*  f \|^2_{L^2(F_\star)} \\
		& \le C \left( \| \varphi \|_{L^2(F_\star)}^2 + \| \nabla_x \varphi \|^2_{L^2(F_\star)} + \| f \|_{L^2(F_\star)}^2 +  \| \nabla_v f \|_{L^2(F_\star)}^2 \right) \, .
	\end{split}
	\end{equation*}
	\end{lem}
	
	\begin{proof}
	The equation satisfied by $\nabla^*_x f$ is obtained from the identities
	\begin{gather*}
	\left[ \nabla^*_x, T \right]f = - \left(\nabla_x^2 V_\star\right) \nabla_v^* + \left(\nabla_x V_\star\right) \left(v \cdot \nabla_x  \psi^\text{e}_{f} \right), \\
	\left[\nabla_{x}^*, (v \cdot \nabla_x \psi^\text{o}_{(\cdot)}) \right]f= \left(\nabla_x V_\star\right) \left(v \cdot \nabla_x  \psi^\text{o}_{f} \right) \quad \text{and} \quad \left[ \nabla_x^*, L \right] = 0 \, .
	\end{gather*}
	Since up to the remainder terms $X_1$ and $X_2$, 
	$\nabla_x^*f$ follows the same equation as $f$, with the source $\nabla_x^*\varphi$, one can follow the proof of Lemma \ref{lem:Lyapunov_order_0} to get (for the same constants) $C, \lambda_0$ and $\mu$ depending on $\underline{\kappa}^\text{e}, \overline{\kappa}^\text{o}, \overline{\kappa}^\text{e}, C_\star$ and $\nu$ that
	\[
	\frac{\d}{\d t} \EE_0\left[ \nabla_x^* f \right] +  \lambda_0 \EE_0\left[ \nabla_x^* f \right] + \mu \| \nabla_x^* \nabla_v  f \|^2_{L^2(F_\star)} \le C \left( \| \nabla_x^* \varphi \|^2_{L^2(F_\star)}  + Y_1 + Y_2 \right) \, ,
	\]
	where, using the boundedness of the operator $A$ and the norm equivalence between $\|\cdot\|$ and $\Nt{\cdot}$,  the remainder terms can be estimated by 
	\begin{align*}
	Y_1&\lesssim \lla\nabla_x^*f, X_1\rra \lesssim \|\nabla_x^*f\|_{L^2(F_\star)}\|(\nabla_x^2 V_\star) \nabla_v^* f\|_{L^2(F_\star)},\\
	Y_2&\lesssim \lla\nabla_x^*f, X_2\rra\lesssim \|\nabla_x^*f\|_{L^2(F_\star)}\|(\nabla V_\star) \left( v \cdot \nabla_x \psi_f \right)\|_{L^2(F_\star)}.
	\end{align*}	
	Then using Young's inequality and  \eqref{eq:poincare_hessianVstar}, one obtains
	\[
	Y_1\lesssim \delta\|\nabla_x^*f\|_{L^2(F_\star)}^2 + \frac{1}{\delta}\left( \eps\|\nabla_x\nabla_v^* f\|^2_{L^2(F_\star)} + C_\eps\|\nabla_v^* f\|^2_{L^2(F_\star)}\right).
	\]	
	Then up to changing the constants, we have by the equivalence \eqref{eq:twisted_spatial_H1_equivalence} and \eqref{eq:charac_norm_dvstar}
	\[
	Y_1\lesssim \delta\|\nabla_x^*f\|_{L^2(F_\star)}^2 + \frac{\eps}{\delta} \left(\|\nabla_x^* f\|^2_{L^2(F_\star)} +\|\nabla_x^*\nabla_v f\|^2_{L^2(F_\star)}\right) + \frac{C_\eps}{\delta}\left(\|f\|^2_{L^2(F_\star)}+\|\nabla_v f\|^2_{L^2(F_\star)}\right).
	\]
	For the second remainder term one can use Hölder's inequality for some $\frac{1}{2} = \frac{1}{q} + \frac{1}{s}$ followed by \eqref{eq:upper_bound_grad_psi_weight}:
	$$
	\begin{aligned}Y_2 &\lesssim \delta\|\nabla_x^*f\|_{L^2(F_\star)}^2 + \frac{1}{\delta}\| | \nabla_x V_\star |^2 \rho_\star \|_{L^{ \frac{s}{2} } } \| \nabla_x \psi_f \|^2_{L^q}\\
	&\lesssim  \delta\|\nabla_x^*f\|_{L^2(F_\star)}^2 +  \frac{\overline{\kappa}^{2}}{\delta}\| \Pi f \|^2_{L^2(F_\star)} \, .\end{aligned}$$
	Putting these estimates together with $\eps = \delta^2$, and using the fact that $\EE_0$ is equivalent to $\| \cdot \|_{L^2(F_\star)}$, we obtain 
	\begin{align*}
	\frac{\d}{\d t} \EE_0\left[ \nabla_x^* f \right] + & ( \lambda_0 - C (\delta+\delta^2)) \EE_0\left[ \nabla_x^* f \right] + ( \mu - C \delta)\| \nabla_x^* \nabla_v  f \|^2_{L^2(F_\star)} \\
	& \le C_\delta \left( \| \nabla_x^* \varphi \|^2_{L^2(F_\star)}  + \overline{\kappa} \| \Pi f \|^2_{ L^2(F_\star) } + \|f\|^2_{ L^2(F_\star) }+ \|\nabla_vf\|^2_{ L^2(F_\star) } \right) \, ,
	\end{align*}
	thus, taking $\delta$ small enough, we conclude the proof.
	\end{proof}

	Define $\EE_{1,0}$ for some small $\eps \in (0, 1)$ as
	$$\EE_{1,0}[f] = \EE_0[f] + \eps \, \EE_0[\nabla_x^* f]\, .$$
	By virtue of \eqref{eq:twisted_spatial_H1_equivalence}, it is clear that $\EE_{1,0}$ defines an equivalent norm to $H^1_x L^2_v (F_\star)$, thus, the following hypocoercivity result follows from Lemmas \ref{lem:Lyapunov_order_0} and \ref{lem:estimate_twisted_gradx} for a small enough choice of $\eps$. More precisely, for such an $\eps$, the functional $\EE_{1,0}$ satisfies a differential inequality of the same form as in Lemma~\ref{lem:Lyapunov_order_0}, namely
	$$\frac{\d}{\d t} \EE_{1,0}[f] + \lambda_0 \, \EE_{1,0}[f] + \mu \| \nabla_v f \|^2_{H^1_x L^2_v(F_\star)} \le C \| \varphi \|^2_{H^1_x L^2_v(F_\star)}$$
	for some $\lambda_0, \mu, C > 0$, the terms $\| f \|^2_{L^2(F_\star)}$ and $\| \nabla_v f \|^2_{L^2(F_\star)}$ in the right-hand side of Lemma~\ref{lem:estimate_twisted_gradx} being absorbed by the dissipation of Lemma~\ref{lem:Lyapunov_order_0}.
	
	\begin{prop}[$H^1_x L^2_v$--hypocoercivity]
	\label{prop:H1x_linear_hypocoercivity}
	Under the smallness assumptions
	$$\underline{\kappa}^\text{e} < \delta^\text{e}(C_\star,\theta) \qquad \text{and} \qquad \overline{\kappa}^\text{o}< \delta^\text{o}(C_\star, \overline{\kappa}^\text{e}, \underline{\kappa}^\text{e}, \nu )$$
	there exists $C$ depending on $C_\star, \overline{\kappa}^\text{e}, \overline{\kappa}^\text{o}, \underline{\kappa}^\text{e}$ and $\nu$ such that, denoting by $\lambda_0$ the rate of the above differential inequality for $\EE_{1,0}$, for any $\lambda \in [0, \lambda_0/2]$, any solution to \eqref{eq:linearized_source} satisfies
	\[
	\sup_{t \ge 0} e^{2 \lambda t} \| f(t) \|_{H^1_x L^2_v (F_\star)}^2 +  \int_0^\infty e^{2 \lambda t} \| \nabla_v f(t) \|_{H^1_x L^2_v (F_\star)}^2 \d t  \le C \left( \| f_\ini \|_{H^1_x L^2_v (F_\star)}^2 + \int_0^\infty e^{2 \lambda t} \|\varphi(t) \|_{H^1_x L^2_v (F_\star)}^2 \d t \right) \, .
	\]
	\end{prop}

	\section{Non-linear VFP}
	\label{scn:nonlinear}
	
	Now we turn to the proof of our main result Theorem \ref{thm:stability}. First, based on the analysis of the linearized equation, we introduce norms which will encompass the gain of regularity of the solution and the exponential decay towards equilibrium. Then we turn to the estimate of the nonlinear quadratic contribution 
	\[
	\varphi = -f\nabla_x\psi_f,
	\]
	in these functional spaces. From there we prove our main results of perturbative well-posedness and asymptotic stability using a contraction argument.
	
	\subsection{Functional setting}	
	\label{scn:functional_setting_interpolation}
	
	   Let $\lambda_0>0$ be the smallest of the rates given by Lemmas~\ref{lem:Lyapunov_order_0} and~\ref{lem:Lyapunov_order_1} and by the differential inequality for $\EE_{1,0}$ stated before Proposition~\ref{prop:H1x_linear_hypocoercivity}. We fix once and for all some $\lambda \in (0, \lambda_0/2)$. The strict inequality ensures that the time integrals involved in the norms below are finite (for instance $\int_0^\infty e^{2\lambda t}\| f(t) \|^2_{L^2(F_\star)}\d t$), and $\lambda > 0$ is used in the proof of Lemma~\ref{lem:nonlinear_bound}. All the constants below depend on this choice of $\lambda$. As in the previous section, we consider the time weights
	$$w_\sigma(t) = e^{\lambda t} \min\{ 1, t \}^{ \frac{\sigma}{2} }, \qquad \sigma \in [0, 3]$$
	and, for any $s\in [0, 1]$, we introduce the norms
	\[
	\| \varphi \|_{\HH^s}^2 = \int_0^\infty \Big( w_0(t)^2 \| \varphi(t) \|_{L^2_{x, v}(F_\star) }^2  +
	w_{1-s}(t)^2 \| \varphi(t)\|_{L^2_x H^{1-s}_v(F_\star)}^2 \\
	 + w_{3 (1-s)}(t)^2\| \varphi(t) \|_{H^1_x L^2_v(F_\star)}^2 \Big) \, \d t
	\]
	as well as
	\[
	\| f \|_{\bXX^s}^2 :=  \sup_{t > 0} \Big( w_0(t)^2 \| f(t) \|_{H^s_x L^2_v (F_\star) }^2
	+ w_{3 (1-s)}(t)^2 \| f(t) \|_{ H^{1}_x L^2_v (F_\star) }^2 
	+ w_{1-s}(t)^2 \| f(t) \|_{ L^2_x H^{1-s}_v (F_\star) }^2 \Big)
	+ \| f \|_{\bHH^s}^2
	\]
	where we denoted
	$$\| f \|^2_{ \bHH^s} := \int_0^\infty \Big( w_{2 (1-s)}(t)^2 \| f(t) \|_{ H^{1}_x L^2_v (F_\star) }^2
	+ w_0(t)^2 \| f(t) \|_{L^2_x H^{1-s}_v (F_\star) }^2 \Big) \, \d t \, .$$
	    Note that the bold symbols $\bXX^s$ and $\bHH^s$ denote norms of the unknown $f$, whereas $\HH^s$ denotes the norm of the source $\varphi$.
	We can now reformulate Proposition \ref{prop:H1x_linear_hypocoercivity} and Proposition~\ref{prop:H1_linear_hypoellipticity_hypocoercivity}, by the fact that 
	any solution $f$ to \eqref{eq:linearized_source} satisfies 
	$$\| f \|_{\bXX^s} \le C \| (f_\ini, \varphi) \|_{H^s_x L^2_v(F_\star) \times \HH^s} \, \quad\text{for }s=0\text{ or }1 \, .$$
	It is possible to extend these bounds to any $s \in [0,1]$. Indeed, according to \cite[Theorem 1.1]{mccarthy}, this follows from the fact that $\HH^s$ (resp. $\bHH^s$) is the geometric interpolation of order $s \in [0, 1]$ between $\HH^0$ and $\HH^1$ (resp. $\bHH^0$ and $\bHH^1$). This can be seen by writing
	$$\| \varphi \|_{\HH^{1-s}} = \| ( \varphi, \varphi, \varphi ) \|_{ [ \HH ]^{s}} = \| A^s ( \varphi, \varphi, \varphi ) \|_{ [ \HH ]^0}$$
	where we denoted the Hilbert norm $[\HH]^0$
	$$\| (\varphi_1, \varphi_2, \varphi_3) \|^2_{ [\HH]^0 } = \| \varphi_1 \|^2_{L^2_{t,x,v} (F_\star e^{2 \lambda t}) } + \| \varphi_2 \|^2_{L^2_{t,x,v} (F_\star e^{2 \lambda t}) } + \| \varphi_3 \|^2_{L^2_{t, v} H^1_x (F_\star e^{2 \lambda t}) }$$
	and the operator $A$
	$$A( \varphi_1, \varphi_2, \varphi_3 ) = \left( \varphi_1 , \min\{1, t\}^{1/2} (\nabla_v^* \nabla_v)^{1/2} \varphi_2 , \min\{1, t\}^{3/2} \varphi_3 \right) \, .$$
	Arguing similarly for $\bHH^s$, one gets by interpolation using \cite[Theorem 1.1]{mccarthy}
	$$\| f \|_{\bHH^s} \le C  \| (f_\ini, \varphi) \|_{H^s_x L^2_v(F_\star) \times \HH^s} \, \quad\text{for } s \in [0, 1] \, .$$
	The pointwise bounds are obtained through the same argument, but applied to any given time $t \in (0, \infty)$. For instance, to obtain the $x$--derivatives control, we have to check that for $s \in [0, 1]$ and $t \in (0, \infty)$
	\begin{gather*}
	\| f(t) \|_{H^1_x L^2_v(F_\star)} \le C_s(t) \| (f_\ini, \varphi) \|_{H^s_x L^2_v(F_\star) \times \HH^s} \, ,
	\end{gather*}
	where we denoted
	$$C_s(t) = C w_{3(1-s)}(t)^{-1} = C_0(t)^{1-s} C_1(t)^s \, .$$
	We know this holds for $s = 0$ and $s=1$, thus, by interpolation, this holds for $s \in [0, 1]$. The same argument provides the $v$--derivatives estimate and the $L^2_{x, v}$--exponential decay.  The following proposition sums up the bounds established by interpolation in this paragraph.

	\begin{lem}
	\label{lem:linear_bound}
	For any given $s \in [0, 1]$, any solution $f$ to \eqref{eq:linearized_source} satisfies
	\begin{equation*}
	\| f \|_{\bXX^s} \le C \left( \| f_\ini \|_{H^s_x L^2_v(F_\star)} + \| \varphi \|_{\HH^s} \right) \, .
	\end{equation*}
	\end{lem}

	\subsection{Nonlinear estimates}
	\label{scn:nonlinear_estimate}
	We now show that these norms are adapted to the nonlinear problem.
	\begin{lem}
	\label{lem:nonlinear_bound}
	For any $s \in [0, 1]$ such that
	$$s > s_c := \frac{3}{2} \left( \frac{d}{q} - \frac{1}{3} \right) \, ,$$
	there exists some $C>0$ such that the following bilinear estimate holds:
	$$\| f \nabla_x \psi_g \|_{\HH^{s}} \le C \| f \|_{\bXX^s} \| g \|_{\bXX^s} \, .$$
	\end{lem}
	
	\begin{rem}
	\label{rem:regularization_K_vs_smoothness_ID}
	Note that $s_c < 1 \Leftrightarrow q > d$, and $s = 0$ is allowed as soon as $q > 3d$.
	\end{rem}
	
	\begin{proof}
	Let us denote in this proof $\sigma := 1-s$ and define
	$$\delta = \frac{2}{3} s - \left( \frac{d}{q} - \frac{1}{3} \right)  > 0
	\quad \text{so that} \quad \frac{d}{q} + \delta = s + \frac{\sigma}{3} \, .$$
	We start with the control
	$$\| f \nabla_x \psi_g \|_{H^1_x L^2_v (F_\star)} \lesssim A + B  \, .$$
	where we denoted
	$$A := \| f \nabla_x \psi_g \|_{L^2_x L^2_v (F_\star)} + \| \nabla_x f \nabla_x \psi_g \|_{L^2_x L^2_v (F_\star)}, \quad B := \| f \nabla_x^2 \psi_g \|_{L^2_x L^2_v (F_\star)} \, .$$
	On the one hand, one has from the bound \eqref{eq:upper_bound_grad_psi_weight} on $\nabla_x \psi$
	$$\| \nabla_x \psi_g \|_{L^q} \lesssim \| g \|_{L^2(F_\star)} \quad \text{and} \quad \| \nabla_x \psi_g \|_{W^{1,q}} \lesssim \| g \|_{H^1_xL^2_v(F_\star)} \, ,$$
	therefore, since we have by real interpolation with $\theta = \frac{d}{q} + \delta = s + \frac{\sigma}{3} \in [0, 1]$ (see for instance \cite[Lemma 22.3]{tartar})
	$$\left[ L^q, W^{1,q} \right]_{\theta, 2} = B^{\theta}_{q, 2} \quad \text{and} \quad \left[L^2(F_\star) , H^1_xL^2_v(F_\star) \right]_{\theta, 2} = H^\theta_xL^2_v(F_\star) \, ,$$
	where $B^\theta_{q,2}$ denoted the suitable Besov space, the following interpolated bound holds:
	$$\| \nabla_x \psi_g \|_{ B^{\theta}_{q, 2} } \lesssim \| g \|_{H^\theta_xL^2_v(F_\star)} \, .$$
	Using the chain of embeddings (see \cite[Propositions 2.71 and then 2.39]{BahouriCheminDanchin})
	$$B^{\theta}_{q, 2} = B^{\frac{d}{q} + \delta}_{q, 2} \subset B^{\delta}_{\infty, 2} \subset B^0_{\infty, 1} \subset L^\infty \, ,$$
	these two facts yield the estimate
	$$A \lesssim \| f \|_{H^1_x L^2_v(F_\star)} \| \nabla_x \psi_g \|_{L^\infty_x } \lesssim \| f \|_{H^1_x L^2_v(F_\star)} \| g \|_{H^{ s + \frac{\sigma}{3} }_x L^2_v(F_\star)} \, .$$
	On the other hand, using Hölder's inequality with $\frac{1}{2} = \frac{1}{q} + \frac{1}{m}$ followed by the bound on $\nabla_x \psi$, one has
	$$B = \| f \nabla_x \psi [\nabla_x^* g] \|_{L^2_x L^2_v(F_\star)} \lesssim \| F_\star^{1/2} f \|_{L^m_x L^2_v} \| \nabla_x^* g  \|_{L^2_x L^2_v(F_\star)} \, .$$
	Using next the Sobolev embedding
	$$H^{s + \frac{\sigma}{3} } = H^{\frac dq +\delta }\subset H^{\frac dq} \subset L^m \quad \text{since} \quad \frac{d}{q} = \frac{d}{2} -\frac{d}{m} \, ,$$
	we thus have thanks to the norm comparison \eqref{eq:weighted_sobolev_space_comparison} that
	$$B \lesssim \| f \|_{H^{s + \frac{\sigma}{3}}_x L^2_v(F_\star) } \| g \|_{H^1_x L^2_v(F_\star)} \, .$$
	Putting together the estimates on $A$ and $B$, using that $s + \sigma = 1$, we conclude that
	\begin{align*}
	t^{3 \sigma} \| f \nabla_x \psi_g \|_{H^{s+\sigma}_x L^2_v(F_\star)}^2 \lesssim & \left( t^{2\sigma} \| f \|_{H^{s+\sigma}_x L^2_v(F_\star)}^2 \right) \left( t^\sigma \| g \|_{H^{s + \frac{\sigma}{3}}_x L^2_v(F_\star) }^2 \right) \\
	& + \left( t^\sigma \| f \|_{H^{s + \frac{\sigma}{3}}_x L^2_v(F_\star) }^2 \right) \left( t^{2\sigma} \| g \|_{H^{s+\sigma}_x L^2_v(F_\star)}^2 \right) \, ,
	\end{align*}
	from which we deduce, interpolating $H^{s + \frac{\sigma}{3} }_x$ between $H^s_x$ and $H^{s+\sigma}_x = H^1_x$
	\begin{align*}
	\int_0^\infty w_{3 \sigma}(t)^2 & \| f(t) \nabla_x \psi_g(t) \|_{H^{s+\sigma}_x L^2_v(F_\star)}^2 \d t \\
	\lesssim & \left( \int_0^\infty w_{2\sigma}(t)^2 \| f \|_{H^{s+\sigma}_x L^2_v(F_\star)}^2 \d t \right)
	\sup_{t > 0} \left( w_\sigma(t)^2 \| g \|_{H^{s + \frac{\sigma}{3}}_x L^2_v(F_\star) }^2 \right) \\
	& + \sup_{t > 0} \left( w_\sigma(t)^2 \| f \|_{H^{s + \frac{\sigma}{3}}_x L^2_v(F_\star) }^2 \right) \left( \int_0^\infty w_{2\sigma}(t)^2 \| g \|_{H^{s+\sigma}_x L^2_v(F_\star)}^2 \d t \right) \\
	\lesssim & \| f \|_{\bXX^s}^2 \| g \|_{\bXX^s}^2
	\end{align*}
	Similarly, one shows
	$$t^\sigma \| f \nabla_x \psi_g \|_{L^2_x H^\sigma_v(F_\star)}^2 \lesssim \| f \|_{L^2_x H^\sigma_v(F_\star)}^2 \left( t^\sigma \| g \|_{H^{s + \frac{\sigma}{3}}_x L^2_v(F_\star) }^2 \right)$$
	from which we deduce
	$$\int_0^\infty w_\sigma(t)^2 \| f(t) \nabla_x \psi_g(t) \|_{L^2_x H^\sigma_v(F_\star)}^2 \d t \lesssim \| f \|^2_{\bXX^s} \| g \|^2_{\bXX^s} \, .$$
	        Finally, the first term of the $\HH^s$--norm is controlled using only the pointwise parts of the $\bXX^s$--norms. Let $\theta \in (\frac{d}{q}, 1]$ be such that $s \le \theta < s + \frac{1}{3}$, which is possible since $\frac{d}{q} < 1$ and $\frac{d}{q} - s < \frac{1}{3}$ by the assumption $s > s_c$. As for the term $A$, one has $\| \nabla_x \psi_g \|_{L^\infty_x} \lesssim \| g \|_{H^\theta_x L^2_v(F_\star)}$, and interpolating $H^\theta_x$ between $H^s_x$ and $H^1_x$ yields
	$$\| g(t) \|_{H^\theta_x L^2_v(F_\star)}^2 \lesssim e^{-2\lambda t} \min\{1,t\}^{-3(\theta - s)} \| g \|_{\bXX^s}^2 \, .$$
	Since $w_0(t) \| f(t) \|_{L^2(F_\star)} \le \| f \|_{\bXX^s}$, we deduce
	$$\int_0^\infty w_0(t)^2 \| f(t) \nabla_x \psi_g(t) \|_{L^2(F_\star)}^2 \d t \lesssim \| f \|_{\bXX^s}^2 \| g \|_{\bXX^s}^2 \int_0^\infty e^{-2\lambda t} \min\{1,t\}^{-3(\theta-s)} \d t \lesssim \| f \|_{\bXX^s}^2 \| g \|_{\bXX^s}^2 \, ,$$
	because $3(\theta - s) < 1$ and $\lambda > 0$. 
	This concludes the proof.
	\end{proof}
	
	\subsection{Proof of Theorem~\ref{thm:stability}}
	\label{scn:nonlinear_proof}
	
	Fix some $f_\ini \in H^s_x L^2_v (F_\star)$ and denote $\Phi : \bXX^s \to \bXX^s$ the mapping defined by
	\begin{equation*}
	f := \Phi g \quad \text{where} \quad
	\begin{cases}
	( \partial_t + \Lambda ) f(t) + v \cdot \nabla_x \psi_f(t) = -\nabla_v^* ( g \nabla_x \psi_g ), \\
	\displaystyle f(0, x, v) = f_\ini(x, v) \, .
	\end{cases}
	\end{equation*}
	As a consequence of Lemmas \ref{lem:linear_bound} and \ref{lem:nonlinear_bound}, there exists some $C > 0$ such that
	$$\| \Phi(g_1) - \Phi(g_2) \|_{\bXX^s} \le C\left( \| g_1 \|_{\bXX^s} + \| g_2 \|_{\bXX^s}\right) \| g_1 - g_2 \|_{\bXX^s} \, ,$$
	$$\| \Phi(g) \|_{\bXX^s} \le C \left(\| f_\ini \|_{H^s_x L^2_v(F_\star) } + \| g \|_{\bXX^s}^2\right) \, .$$
	It then follows from a classical Picard argument that for $\| f_\ini \|_{H^s_x L^2_v(F_\star)}$ small enough, the mapping $\Phi$ has a unique fixed point  $f \in \bXX^s$. It remains to prove uniqueness under the assumption of Theorem \ref{thm:stability}. Consider another solution $g$ such that $\nabla_x \psi_g \in L^2_{t, \text{loc}} L^\infty_x$, then the difference $h = g - f$ satisfies
	    $$\partial_t h + \Lambda h + v \cdot \nabla_x \psi_h (t) = -\nabla_v^* (h \nabla_x \psi_{g} + f \nabla_x \psi_h) \ , \qquad h(0) = 0 \ .$$
	    Performing a naive energy estimate, and using the same arguments as in the previous sections, one gets with $\frac{1}{2} = \frac{1}{q} + \frac{1}{m}$
	    \begin{align*}
	        \frac{\d}{\d t} \| h \|^2_{L^2(F_\star)} \lesssim & \| h \|^2_{L^2(F_\star)} + \| h \nabla_x \psi_g \|^2_{L^2(F_\star)} + \| f \nabla_x \psi_h \|^2_{L^2(F_\star)} \\
	        \lesssim & \| h \|^2_{L^2(F_\star)} \left( 1 + \| \nabla_x \psi_{g} \|_{L^\infty_x}^2 + \| F_\star^{1/2} f \|_{L^m_x L^2_v}^2 \right) \ .
	    \end{align*}
	    As in the proof of Lemma \ref{lem:nonlinear_bound}, using a Sobolev embedding and hypoellipticity (note that $d/q \in [0, 1)$)
	    $$\| F_\star^{1/2} f \|_{L^m_x L^2_v}^2 \lesssim \| f \|_{H^{d/q}_x L^2_v(F_\star) }^2 \lesssim t^{-3 \left(\frac{d}{q} - s\right)_+} \| f_\ini \|_{H^{s}_x L^2_v(F_\star) }^2 \ , \qquad 0 < t \le 1 \ .$$
	    This singularity is indeed integrable near $t=0$:
	    \begin{itemize}
	        \item if $q > 3d$, then $0 < \frac{d}{q} < \frac13$,
	        \item if $q \le 3 d$, since $s > s_c$, we have $\frac{d}{q} - s < \frac{1}{2} \left( 1 - \frac{d}{q} \right) \le \frac13$.
	    \end{itemize}
	    We thus conclude by a classical Gronwall argument that $h=0$ and thus $f$ is the unique solution. Since one proves similarly $\| \nabla_x \psi_f \|_{L^\infty} \lesssim \| f \|_{H^{d/q+\delta} }$ for $\delta>0$ arbitrarily small, we also conclude that $\nabla_x \psi_f \in L^2_{t, \text{loc}} L^\infty_x$.

	\section{Concluding remarks}
	\label{scn:perspectives}
	Let us end this work by stating some perspectives and natural continuations. We first present possible improvements on the regularity assumptions.
	
	\paragraph{Weakly regularizing interaction potential}
	We have assumed the interaction potential to be regularizing enough, namely $\nabla \KK : L^1 \cap L^2 \to L^q$ with $q > d$, so that the initial condition only needs to have (at most) $H^1_x$ regularity. In order to extend Theorem \ref{thm:stability} to the case when we only assume $q \ge 2$ (for instance, Manev interactions, see Example \ref{exam:more_singular}), one would need to consider an initial data $H^s_x$ with $s$ big enough. The strategy to derive $H^1_x L^2_v$ estimates from Section \ref{scn:H1xL2v_hypocoercivity} would need to be adapted to $H^k_x L^2_v$ for any integer $k \in \N$, thus requiring suitable adaptation of the hypothesis \eqref{eq:smoothness_V}. The nonlinear estimate would need to be modified, for instance in the non-regularizing case $q=2$, the Sobolev algebra inequality \cite[Corollary 2.86]{BahouriCheminDanchin} yields for $s > \frac{d}{2} - \frac{1}{3}$
	$$t^3 \| f \nabla_x \psi_f \|_{H^{s+1}_x L^2_v(F_\star)} \lesssim t^2 \| f \|_{ H^{s+1}_x L^2_v(F_\star) }^2 \times t \| f \|_{H^{s+\frac{1}{3} }_x L^2_v(F_\star) }^2 \, .$$
	
	\paragraph{Rough initial data}
	Similarly, we have only considered initial data with at least $0$ regularity, in the sense that we excluded $H^s_x$ spaces for $s < 0$. However, when $q > 3 d$, the critical regularity index is $s_c < 0$, which suggests that initial data with regularity $H^s_x$ for $s_c < s < 0$ could be considered. This would require a duality argument in the spirit of \cite{herau_2004_isotropic, herau_tonon_tristani, carrapatoso2022regularization, gervais}.
	
	\bigskip
	Let us now present some natural continuations.
	\bigskip

	\paragraph{Non-perturbative setting}
	For VPFP, one would be interested in showing exponential convergence to equilibrium starting from an initial data $F(0)$ assuming only physical bounds, namely finite entropy and energy. The strategy would then be to use the free energy, together with regularization estimates in order to perform a trapping argument and show that after a (possibly non-explicit) time $t_0 > 0$, the solution is close to the steady state, making it possible to conclude using Theorem \ref{thm:stability} with $F_\ini = F(t_0)$. Similar strategies have been considered in \cite{mouhot, tristani, gualdani2017factorization} for spatially homogeneous kinetic equations. Let us also mention works concerning the global convergence to equilibrium in the case of Vlasov-Fokker-Planck equation \cite{bouchut_1995_long, guillin_2022_convergence}.
	
	\paragraph{Phase transitions and instabilities}
	It is known, in the case of the torus \cite{carrillo_2020_long}, that negative Fourier modes of the interaction potential may lead to phase transitions, and more precisely non-uniqueness of steady states. In particular, in the case of the synchrotron model \eqref{eq:synchrotron}, experimental and numerical evidence suggest that instabilities may arise, and an asymptotically periodic behavior seem to appear (see \cite{roussel2014spatio}). In the linearized setting, it has been shown that when the size of $k_S$ is large, oscillating modes appear \cite{cai2011linear}. This is to be contrasted with the symmetric positive setting (such as Poisson interactions) for which global convergence to equilibrium is known to hold.

	\paragraph{Acknowledgements}The authors thank the anonymous referee for constructive remarks that improved the quality of the manuscript. The authors acknowledge support from the LabEx CEMPI (ANR-11-LABX-0007) and CPER WaveTech@HdF.

	    \paragraph{Declaration of AI use} The authors used the AI assistant Claude (Anthropic) during the revision of this manuscript, to double-check proofs and detect typos and inconsistencies. All suggestions were verified by the authors who take responsibility for all content.
	
	\paragraph{Data availability statement}	
	No datasets were generated or analyzed during the current study
	
	\bibliographystyle{plain}
	\bibliography{bibli}	
	\end{document}